\documentclass[reqno,11pt]{amsart}
\usepackage{amsmath,amssymb,amsthm,mathrsfs}
\usepackage{epsfig,color}
\usepackage[latin1]{inputenc}
\usepackage[english]{babel}
\usepackage[numbers]{natbib}

\usepackage{esint}

\usepackage{hyperref}

 \allowdisplaybreaks
\numberwithin{equation}{section}

\def\R{\mathbb R}

\newcommand{\dist}{\mathop{\mathrm{dist}}}

\def\dist{{\rm dist}}

\def\00{{\bf 0}}

\newcommand{\diver}{{\rm div}}

\newcommand{\ric}{{\rm Ric}}

\newtheorem*{theorem*}{Theorem}
\newtheorem{theorem}{Theorem}[section]
\newtheorem{lemma}[theorem]{Lemma}
\newtheorem{proposition}[theorem]{Proposition}
\newtheorem*{proposition*}{Proposition \ref{prop_stime_cgen}}

\newtheorem{remark}[theorem]{Remark}

\begin{document}
  
    \title[]{Non-existence of patterns for a class of weighted degenerate operators}
    
  \date{}

\author{Giulio Ciraolo}
\address{G. Ciraolo. Dipartimento di Matematica "Federigo Enriques",
Universit\`a degli Studi di Milano, Via Cesare Saldini 50, 20133 Milano, Italy}
\email{giulio.ciraolo@unimi.it}

\author{Rosario Corso}
\address{R. Corso. Dipartimento di Matematica e Informatica,
	Universit\`a degli Studi di Palermo, Via Archirafi 34, 90123 Palermo, Italy
}
\email{rosario.corso02@unipa.it}

\author{Alberto Roncoroni} 
\address{A. Roncoroni.  Dipartimento di Matematica, Politecnico di Milano, Piazza Leonardo da Vinci 32, 20133, Milano, Italy}
\email{alberto.roncoroni@polimi.it}


\begin{abstract} 
A classical result by Casten-Holland and Matano asserts that constants are the only positive and stable solutions to semilinear elliptic PDEs subject to homogeneous Neumann boundary condition in bounded convex domains. In other terms, this result asserts that \emph{stable patterns} do not exist in convex domains.

In this paper we consider a weighted version of the Laplace operator, where the weight may be singular or degenerate at the origin, and prove the nonexistence of patterns, extending the results by Casten-Holland and Matano to general weak solutions (not necessarily stable) and under a suitable assumption on the nonlinearity and the domain. 

Our results exhibit some intriguing behaviour of the problem according to the weight and the geometry of the domain. Indeed, our main results follow from a geometric assumption on the second fundamental form of the boundary in terms of a parameter which describes the degeneracy of the operator. As a consequence, we provide some examples and show that nonexistence of patterns may occurs also for non convex domains whenever the weight is degenerate.
\end{abstract}

\maketitle

\section{Introduction}
Let $\Omega \subset \mathbb{R}^d$, $d \geq 2$, be a bounded domain and consider the problem
\begin{equation} \label{eq_Laplace}
\begin{cases}
\Delta u + f(u) = 0 & \text{ in } \Omega \,, \\
u_\nu = 0 & \text{ on } \partial \Omega \,,
\end{cases}
\end{equation}
where $\nu$ is the outward normal to $\Omega$.

A classical result of Casten-Holland \cite{CastenHolland} and Matano \cite{Matano} states that all stable solutions of \eqref{eq_Laplace} are constant in bounded convex domains. In other words, by calling \emph{pattern} a non-constant solution of \eqref{eq_Laplace}, this result asserts that  stable patterns do not exist in convex domains. Apart from its own mathematical interest, this result has relevant consequences in the classification of solutions, in the study of asymptotics of the associated evolution problems and it is motivated by applications in chemistry, population dynamics, and many others (see \cite[Section 3]{Nordmann} for an interesting and detailed discussion). 

This result has also been extended in several directions, by considering nonlinear elliptic operators and other boundary conditions, on manifolds, unbounded or more general domains and also to some type of systems; we refer to \cite{BMMP, BPT, CCR, DPV, DPV2, Jimbo, JM, Nordmann,Nordmann2,Yanagida} and references therein.

In \cite{CCR} we established a result in the spirit of Casten-Holland and Matano by considering a general weak solution of \eqref{eq_Laplace}. More precisely, we removed the stability assumption on the solution and we proved that constants are the only weak solutions to \eqref{eq_Laplace} under the assumption that the nonlinearity satisfies the following condition 
\begin{equation} \label{f_condition}
\frac{f(u)}{u^{\frac{d+2}{d-2}}} \text{ is non-increasing} \,.
\end{equation}
This condition is in some sense optimal since one can construct counterexamples by adding a small linear perturbation to $f$ (see the discussion in \cite[Introduction]{CCR}); this problem is related to the Lin-Ni conjecture \cite{Lin-Ni}.
The results in \cite{CCR} hold for more general operators, in particular we considered nonlinear anisotropic $p$-Laplace type operators, and also more general boundary conditions.

It is also interesting to notice that the technique used in \cite{CCR} can be seen as a spin-off of the one used in \cite{CFR} where, by exploiting suitable integral identities, critical points of Sobolev inequality 
\begin{equation*}
\left (\int_{\R^d} u^{p^*} dx\right)^{\frac{1}{p^*}} \leq C \left (\int_{\R^d} |Du|^p dx\right )^{\frac{1}{p}}
\end{equation*}
were classified. In this setting, the problem boils down to the classification of solutions to critical $p$-Laplace type equations\footnote{By a critical $p$-Laplace type equation we mean an equation of the form $\Delta_p u + u^{p^*-1}=0$, where $p^*$ is the Sobolev exponent related to $p$. We recall that, in the case $p=2$, we have that $2^*-1 = (d+2)/(d-2)$.} and positive solutions are proved to be radially symmetric (up to a translation). In \cite{CCR}, the classification result naturally follows once such integral identities are carefully adapted to the case of a convex bounded domain and if \eqref{f_condition} is in force.

The main goal of this paper is to continue this line of research in the direction of weighted operators. In particular we consider weak solutions to 
\begin{equation} \label{pb_general_intro}
 	\begin{cases}
 	\diver(|x|^{-2a} D u) + |x|^{-bq}f(u) = 0& \text{ in }\Omega \,,\\
 	u_\nu = 0 & \text{ on } \partial \Omega \,,
 	\end{cases}
 	\end{equation}	
where $\Omega \subset \mathbb{R}^d$ is a bounded domain, $d\geq 3$, and $a,b,q \in \mathbb{R}$ are parameters satisfying certain conditions to be specified later. We emphasize that when the origin $O \in \Omega$, the problem has some relevant regularity issues to be carefully considered. It is clear that \eqref{pb_general_intro} reduces to \eqref{eq_Laplace} when $a=b=0$.

This type of weighted equations arise as the Euler-Lagrange equations of Caffarelli-Kohn-Nirenberg (CKN) inequalities 
\begin{equation}\label{CKN}
\left (\int_{\R^d} |x|^{-bq}u^q dx\right)^{\frac{1}{q}} \leq C_{a,b} \left (\int_{\R^d} |x|^{-2a}|Du|^2 dx\right )^{\frac{1}{2}}
\end{equation}
(see \cite{CKN}), where 
\begin{equation} \label{parameters} 
a\leq b <  a+1\,, \ q=\frac{2d}{d-2(1+a-b)}\,, \ a < a_c\,,
\end{equation}
with
\begin{equation} \label{ac_def}
a_c=\frac{d}{2}-1\,,
\end{equation}
and, for future reference, we set
\begin{equation}\label{alpha_def}
\alpha = \frac{(1+a-b)(a_c-a)}{a_c-a+b}
\end{equation}
and 
\begin{equation} \label{n_def}
n=\frac{d}{1+a-b} 
\end{equation}
(notice that $n \geq d$).

CKN inequalities can be seen as interpolation inequalities between Sobolev and Hardy inequalities and they exhibit a very interesting feature about the symmetry of extremals. Indeed, as it was proved in \cite{DEL} and \cite{FS}, extremals of \eqref{CKN} are radially symmetric if and only if 
\begin{equation} \label{alpha_optimal}
\alpha \leq \sqrt{\dfrac{{d-1}}{n-1}} \,.
\end{equation}

Hence, in analogy to what done in \cite{CCR} starting from the argument in \cite{CFR} for critical points of Sobolev inequality, it is natural to investigate whether the symmetry result proved in \cite{DEL} leads to a classification result for \eqref{pb_general_intro} in bounded convex domains. Actually, we started considering equations like \eqref{pb_general_intro} because we were interested in studying how the geometry of the domain influences the geometry of solutions. Indeed, since the operator in \eqref{pb_general_intro} depends on the space variable $x$ and may be degenerate or singular at the origin due to the space variable, it is not clear if the convexity of the domain is the natural assumption to consider in order to have a classification result in the spirit of the results by Casten-Holland and Matano. Moreover we notice that, even in the case $\Omega=B_R$ is a ball of radius $R$ centered at the origin, it is not clear if solutions are radial for any range of the parameters $a$ and $b$ and then the study of radial solutions to \eqref{pb_general_intro} may be of limited interest.

Our main results go in two directions. We first give a classification result for \eqref{pb_general_intro} if $\Omega=B_R$; this result is likely optimal in some sense that we are going to describe later. Then we prove a classification result for a more general $\Omega$; this result will exhibit an interesting feature by showing that the notion of convexity has to be suitably modified in order to get the classification result.

We start by describing our first result, i.e. we consider \eqref{pb_general_intro} with $\Omega = B_R$.

 \begin{theorem}\label{teo2}
 	Let $B_R\subset \mathbb{R}^d$, with $d \geq 3$. Let $f \in C^1([0,+\infty))$ satisfy 
\begin{equation} \label{Phi_def} 
 	\Phi(t):=\frac{f(t)}{t^{\frac{n+2}{n-2}}} \quad \textmd{ is non-increasing, }
 \end{equation}
where $n$ is given by \eqref{n_def}. 
	
	If \eqref{alpha_optimal} holds then there exist no positive bounded weak solutions $u$ to
 	\begin{equation}\label{pb_palla}
 	\begin{cases}
 	\diver(|x|^{-2a} D u) + |x|^{-bq}f(u) = 0& \text{ in }B_R \,,\\
 	u_\nu = 0 & \text{ on } \partial B_R \,,
 	\end{cases}
 	\end{equation}
 	unless $u$ is constant.
 \end{theorem}

We emphasize that the parameter $n$ is (in general) greater than the dimension $d$ and it acts as a new fictitious dimension. It is clear that \eqref{Phi_def} is the same as \eqref{f_condition} in the case $a=b=0$, since in this case $n=d$. Hence, as it was showed in \cite{CCR}, \eqref{Phi_def} can be considered optimal in some sense. Indeed, one can add a small linear perturbation to $f$ and prove that Theorem \ref{teo2} fails, which can be done by using several results on Lin-Ni conjecture (see \cite{Wang_TAMS,WangII,Wei,CCR} and references therein).

Regarding the optimality of the range of the parameters, we recall that condition \eqref{alpha_optimal} gives the optimal region of symmetry for minimizers of CKN inequalities \cite{DEL} (see also \cite{DGZ}). Our approach in Theorem \ref{teo2} deeply use the results in \cite{DEL} and for this reason we conjecture that \eqref{teo2} is optimal. 

When $\Omega$ is not a ball, we have to argue in a different way and the assumption \eqref{alpha_optimal} is too weak to conclude, since many tools that we use in Theorem \ref{teo2} work only in a radial setting. In this case, we have to follow a different strategy and this has the cost of reducing the range of $\alpha$. Moreover, the convexity of $\Omega$ does not seem to be a suitable assumption to conclude and it must be strenghtened as we are going to explain below.

 \begin{theorem}\label{teo1}
Let $\alpha$ satisfy 
\begin{equation}\label{hp_alfa}
 	\alpha < \sqrt{\dfrac{{d-2}}{n-2}},
 	\end{equation}
and let $\Omega\subset \mathbb{R}^d$, $d \geq 3$, be a bounded domain satisfying
\begin{equation} \label{condTeo1} 
\mathrm{II}_{\partial \Omega}\geq (1-\alpha) \frac{x \cdot \nu}{|x|^2}  \,,
\end{equation}
where $\mathrm{II}_{\partial \Omega}$ and $\nu$ denote the second fundamental form of $\partial \Omega$ and the unit outward normal to $\Omega$ at a point $x \in\partial \Omega$.

Let $f \in C^1(\mathbb R)$ satisfy \eqref{Phi_def}.
 	 Then there exist no positive bounded weak solutions $u$ to
 	\begin{equation*} 
 	\begin{cases}
 	\diver(|x|^{-2a} D u) + |x|^{-bq}f(u) = 0& \text{ in }\Omega \,,\\
 	u_\nu = 0 & \text{ on } \partial \Omega \,,
 	\end{cases}
 	\end{equation*} 
 	unless $u$ is constant.
 \end{theorem}
 
We notice that, if $\Omega$ is a ball centered at the origin then it satisfies \eqref{condTeo1} (see Section \ref{sect_remarks}). Hence, in this case, Theorem \ref{teo2} improves Theorem \ref{teo1} since a larger range of $\alpha$ can be considered. On the other hand, in Section \ref{sect_remarks} we show that \eqref{condTeo1} is not fulfilled by any ball in $R^d$. In particular, if $\Omega=B_R(x_0)$ is a ball of radius $R$ centered at $x_0 \neq O$, then \eqref{condTeo1} is not satisfied if $O\in B_R(x_0)$ and $\alpha R \leq |x_0|$.  This example motivates the following proposition which give some insights on condition \ref{condTeo1}.

\begin{proposition} \label{propC}
Let $0 < \alpha <1$ and let $\Omega\subset \mathbb{R}^d$, $d \geq 3$, be a bounded domain satisfying \eqref{condTeo1}. If $O \in \Omega$ then $\Omega$ is convex. If instead $O \in \mathbb{R}^d \setminus \overline \Omega$ then there exist both convex and nonconvex domains satisfying \eqref{condTeo1}.
\end{proposition}

Hence, from Theorem \ref{teo1} and Proposition \ref{propC}, we have a result of nonexistence of patterns also for non convex domains. We mention that in \cite{Nordmann} the author provides examples of nonconvex domains for which one has nonexistence of patterns for the classical Laplace operator ($\alpha=1$ without weights).

\subsection{Description of the strategy and organization of the paper}
Theorems \ref{teo2} and \ref{teo1} are consequences of another classification result which is stated in a suitable Riemannian setting. This will be also the occasion to describe the strategy of our approach. 

For $\alpha > 0$, we consider the change of variables 
\begin{equation} \label{T_def}
T:x \mapsto |x|^{\alpha-1} x \,,
\end{equation}
with $\alpha$ given by \eqref{alpha_def}, and we write \eqref{pb_general_intro} in a Riemannian setting by considering a suitable metric $g$. In this way, we obtain that \eqref{pb_general_intro} is equivalent to the study of the following problem
\begin{equation} \label{pb_Riem} 
\begin{cases}
L w + f(w) = 0 & \text{ in }\tilde \Omega \,,\\
g(\nabla w, \nu_g)=0 & \text{ on } \partial \tilde\Omega \,,
\end{cases}
\end{equation}
where $w(T(x))=u(x)$, 
\begin{equation} \label{Omegatilde_def}
\tilde \Omega = T(\Omega)
\end{equation}
and 
\begin{equation} \label{Lv_def}
L v := |x|^{d-n} \diver_g(|x|^{n-d} \nabla v)  \,.
\end{equation}
Here, and in the rest of the paper, $\nabla w = \nabla_g w$ denotes the gradient of $w$ in the Riemannian manifold $(\mathbb{R}^d,g)$, with $g$ given by
\begin{equation} \label{g_def}
g_{ij}=\delta_{ij}+\left (\frac 1{\alpha^2}-1\right )\frac{x_ix_j}{|x|^2} \,.
\end{equation}
In this setting, our  main result is the following.

\begin{theorem} \label{teo_Riem}
Let $w$ be a positive bounded weak solution to \eqref{pb_Riem} and assume that $\tilde \Omega \subset \mathbb{R}^d$, with $d \geq 3$, is bounded and convex with respect to the metric $g$. Let $f \in C^1(\mathbb R)$ satisfy \eqref{Phi_def}.
 	If $\alpha$ satisfies \eqref{hp_alfa} then $w$ is constant.
\end{theorem}

Going back to the Euclidean setting, Theorem \ref{teo1} is a directly consequence of Theorem \ref{teo_Riem}. Indeed, condition \eqref{condTeo1} guarantees the convexity of $T(\Omega)$ with respect to the metric $g$, and we have to consider this condition since the mapping $T(x)=|x|^{\alpha-1} x$ does not preserve convexity (in general). It would be interesting to prove that Theorem \ref{teo1} fails if $\Omega$ is convex and does not satisfies \eqref{condTeo1}.
 
As far as we know, Theorems \ref{teo2}, \ref{teo1} and \ref{teo_Riem} are the first ones in literature where a Casten-Holland-Matano result is obtained for weighted operators having some degeneracy in the space variable. Moreover, we emphasize that the study of this type of weights is well-motivated by the study of CKN inequalities and, at the same time, they introduce remarkable difficulties since they can be degenerate or singular at the origin where standard elliptic theory does not apply. Moreover, as the study of classification of extremals of CKN inequalities reveals, this type of degeneracy has a strong influence on the geometry of the solution and, a priori, it is not clear what is the optimal range of the parameters in order to obtain the desired classification result. 

Regarding the proofs of our main results, we mention that Theorem \ref{teo1} immediately follows from Theorem \ref{teo_Riem}. Theorem \ref{teo2} shares part of the proof with Theorem \ref{teo_Riem}, but the conclusion follows by using a finer argument.

The main idea to prove Theorem \ref{teo_Riem} is the following. After a careful regularity analysis of the solution at the origin (in the case $O \in \tilde \Omega$) and by using the convexity of $\tilde \Omega$, we find that the function $v=w^{-\frac{2}{n-2}}$ satisfies
	\begin{equation}\label{ineq_intro}
	\int_{\tilde \Omega}v^{1-n}\textsf{k}[v]\, \, |x|^{n-d} dx  \leq \frac{n-1}{n}\int_{\tilde \Omega}  v^{-n}|\nabla v|_g^2 \Phi'(v)  \, |x|^{n-d} dx \leq 0
	\end{equation}	
where the last inequality follows from \eqref{Phi_def} and where we set
$$
\textsf{k}[v] = |H_v|^2 -\frac{1}{n}(Lv)^2 + \ric_g(\nabla v,\nabla v) + H(\nabla v,\nabla v) \,,
$$
with $H_v$ and $H$ denoting the Hessian of $v$ and $(n-d)\log|x|$, respectively.
Hence, the conclusion follows if we are able to prove that
$$
\int_{\tilde\Omega}|x|^{n-d}v^{1-n}\textsf{k}[v]\, \, dx \geq 0\,. 
$$ 
This inequality is obtained in a different way for Theorems \ref{teo2} and \ref{teo1}. In Theorem \ref{teo2} we can use that $\Omega=B_R$ is a ball and then we can take advantage of the radial symmetry of the domain (and of the operator) and argue as done in \cite[Corollary 5.4]{DEL}, which makes use of fine integral estimates in the angular component of the solution. Instead, in order to prove Theorem \ref{teo1}, i.e. when $\tilde \Omega$ is a generic convex domain, we are only able to prove the pointwise estimate 
$$
\textsf{k}[v]\geq 0 \,,
$$ 
which again yields the conclusion (but with in a smaller range of the parameters).

\begin{remark} \label{remark_f0}
We remark that the convexity assumption on the domain which appears in Theorems \ref{teo2}, \ref{teo_Riem} and \ref{teo1} can be dropped if one assumes that $f(0) \leq 0$. 

Indeed, if we assume that $f(0) \leq 0$ then from \eqref{Phi_def} we obtain that $f (t) \leq 0$ for any $t \in [0,+\infty)$. By multiplying 
$$
\diver(|x|^{-2a} D u) + |x|^{-bq}f(u) = 0
$$
by $u$ and integrating by parts in $\Omega$, from $u_\nu = 0$ on $\partial \Omega$ we obtain that 
$$
0 \leq  \int_\Omega |x|^{-2a} |D u|^2 dx = \int_\Omega |x|^{-bq}u f(u) dx \leq 0 \,,
$$
where the last inequality follows from the fact that $u f(u) \leq 0$. Hence $Du=0$ in $\Omega$ and then $u$ is constant in $\Omega$. 

For this reason, throughout this paper we will implicitly assume that $f(0)>0$, since the case $f(0) \leq 0$ is trivial.

\end{remark}

\medskip

The paper is organized as follows. In Section \ref{section_identity} we prove the integral identity implying \eqref{ineq_intro}, which is the key ingredient for the proof of Theorems \ref{teo2}, \ref{teo1} and \ref{teo_Riem}. In Section \ref{section_proofs} we give the proof of the main theorems. In Appendix \ref{appendix}, we give some regularity estimates at the origin which are essentially taken from \cite{DEL}.

\section{An integral identity} \label{section_identity}
The main goal of this section is to prove the integral identity implying \eqref{ineq_intro}, which is the key ingredient for the proof of Theorems \ref{teo2} and \ref{teo1}. As already mentioned in the introduction, it will be useful to formulate the problem in a suitable Riemannian setting.

Let $u$ be a solution to 
\begin{equation} \label{eq_u_sec2}
\begin{cases}
\diver(|x|^{-2a} D u) + |x|^{-bq}f(u) = 0& \text{ in }\Omega \,,\\
u_\nu = 0 & \text{ on } \partial \Omega \,,
\end{cases}
\end{equation}
and consider the map $T:\mathbb{R}^d \to \mathbb{R}^d$, with $T(x)= |x|^{\alpha-1} x $, and the function $w$ defined by 
\begin{equation} \label{w_u}
w(T(x)) = u(x)\,,
\end{equation}
with $\alpha$ given by \eqref{alpha_def}; straightforward computations show that $w$ satisfies
\begin{equation} \label{pb_w}
\begin{cases}
L w  + f(w) = 0 & \text{ in }\tilde \Omega \,,\\
g(\nabla w, \nu_g)=0 & \text{ on } \partial \tilde\Omega \,,
\end{cases}
\end{equation}
with
\begin{equation} \label{g_inverse}
g^{ij}=\delta_{ij}+\left (\alpha^2-1\right )\frac{x_ix_j}{|x|^2}\,,
\end{equation}
and where we set
\begin{equation}\label{f-Laplacian_def}
L w:=|x|^{d-n}\diver_g(|x|^{n-d}\nabla w)\, , 
\end{equation}
\begin{equation} \label{tildeOmega_def}
\tilde \Omega= \{x\in \R^n:|x|^{\alpha-1}x\in \Omega\}
\end{equation}
and
\begin{equation}
\nu_g^i(T(y))= \left ((\alpha-1)\frac{y_iy_j}{|y|^2}+\delta_{ij}\right )\nu^j(y) \,,
\end{equation}
for $y\in\partial \Omega$. Here we recall that we denote by $\nabla$ the gradient in the Riemannian manifold $(\mathbb{R}^d, g)$, i.e. $\nabla^i w = g^{ij}D_jw$.
We notice that the weighted operator \eqref{f-Laplacian_def} can be written in the following (useful) way 
\begin{equation}\label{f-Laplacian_bis}
L w=\Delta_g w+(n-d)g(\nabla \log \vert x\vert,\nabla w)\, .
\end{equation}
Moreover, it will be useful to consider the so-called Bakry--Émery--Ricci curvature, which is defined as follows:
\begin{equation}\label{f-Ricci}
\ric_{g}^{f}:=\ric_g + H\,,
\end{equation}
where $H$ denotes the Hessian of $(n-d)\log|x|$, with
\begin{equation}\label{ric_g}
\ric_g(\nabla w,\nabla w)=(1-\alpha^2)\frac{d-2}{|x|^2}\left (|\nabla w|^2-\frac{(\nabla w\cdot x)^2}{|x|^2} \right)
\end{equation}
\begin{equation}\label{H_g}
H(\nabla w,\nabla w)=\alpha^2 \frac{n-d}{|x|^2}\left (|\nabla w|^2-\frac{(\nabla w\cdot x)^2}{|x|^2}\right)-(n-d)\frac{(\nabla w\cdot x)^2}{|x|^4} \,.
\end{equation}

Now, we consider the function $v$ defined by
\begin{equation}
\label{def_v}
v = w^{-\frac{2}{n-2}} \, .
\end{equation}
It is straightforward to verify that $v$ is a solution to
\begin{equation} \label{pb_general}
\begin{cases}
L  v=
\hat f(v)+\dfrac{n}{2}\dfrac{|\nabla v|_g^2}{v}\,  & \text{ in }\tilde \Omega \,,\\
g(\nabla v, \nu_g)=0 & \text{ on } \partial \tilde \Omega \,,
\end{cases}
\end{equation}
where 
\begin{equation} \label{hat_f_v}
\hat f (v) = \dfrac{2}{n-2}f(v^{-\frac{n-2}{2}})v^{\frac{n}{2}}\,.
\end{equation}
The main result of this section is the following proposition.

\begin{proposition} \label{lemma_integr3}
	Let $v$ be given by \eqref{def_v}.  Then we have 
	\begin{multline}\label{prop_integr}
	\frac{n-1}{n}\int_{\tilde \Omega}  |x|^{n-d} v^{-\frac32 n}|\nabla v|_g^2  \Phi'\left(v^{-\frac{n-2}{2}}\right) dx \\ = \int_{\tilde \Omega}|x|^{n-d}v^{1-n}\textsf{\emph{k}}[v]\,	  \, dx 
 + \int_{\partial \tilde \Omega}|x|^{n-d} v^{1-n} \mathrm{II}(\nabla^T v,\nabla^T v) \, dx \,	
\end{multline}	
where $\Phi$ is given by \eqref{Phi_def}, $\mathrm{II}(\cdot,\cdot)$ denotes the second fundamental form of $\partial \tilde \Omega$, $\nabla^T v$ is tangential component of $\nabla v$ on $\partial \tilde \Omega$ and $\textsf{\emph{k}}[v]$ is given by
\begin{equation}\label{def_k}
\textsf{\emph{k}}[v]:= |H_v|^2
	+\ric_g(\nabla v,\nabla v)+ H(\nabla v,\nabla v)-\frac{1}{n}(Lv)^2\,,
\end{equation}
where $\ric_g$ and $H$ are given by \eqref{ric_g} and \eqref{H_g}, respectively.

Moreover, if $\tilde \Omega$ is convex and $\Phi'\leq0$ then 
\begin{equation}\label{prop_integr_ineq}
 \int_{\tilde \Omega}|x|^{n-d}v^{1-n}\textsf{\emph{k}}[v]\, \, dx \leq 0.
\end{equation}	
\end{proposition}

Before giving the proof of Proposition \ref{lemma_integr3}, we give some remark and anticipate some results whose proofs are given in Appendix \ref{appendix}. It is clear that, once \eqref{prop_integr} is proved then \eqref{prop_integr_ineq} immediately follows from the convexity of $\tilde \Omega$ since the second fundamental form of $\tilde \Omega$ is nonnegative definite. 

The proof of Proposition \ref{lemma_integr3} is based on a differential identity, which is proved in the following lemma. Some related identities can be found in \cite{BiCi} and \cite{CC}.

\begin{lemma} \label{lemma_diff_ident_DEL}
Let $v:E \to \mathbb{R}$, with $E \subset \mathbb{R}^d$, $d \geq 1$, and assume that $v\in C^3(E)$. The following differential identity  
	\begin{multline} \label{diff_ident_DEL}
	|x|^{d-n}\diver_g\Bigg(v^{1-n}L v |x|^{n-d}\nabla v + \frac{1-n}{2}v^{-n}|\nabla v|_g^2 |x|^{n-d} \nabla v \\ -\dfrac{1}{2} v^{1-n}|x|^{n-d}\nabla |\nabla v|_g^2+ v^{-n}|x|^{n-d}\frac{2(1-n)}{n(n-2)}\nabla v\Bigg) \\ 
	=-\frac{v}{n}\left (Lv-\frac{n}{2}\frac{|\nabla v|_g^2}{v}-\frac{2}{n-2}\frac{1}{v}\right )L (v^{1-n}) -v^{1-n}\textsf{\emph{k}}[v]
	\end{multline}
	holds, where $\textsf{\emph{k}}[v]$ and $n$ are given by \eqref{def_k} and \eqref{n_def}, respectively.
\end{lemma}

\begin{proof}
From \eqref{f-Laplacian_def} we have 
\begin{multline}\label{1_vera}
|x|^{d-n}\diver_g\left(|x|^{n-d}v^{1-n}L v\nabla v\right)= (1-n) v^{-n}\vert\nabla v\vert^2L v+ v^{1-n}g(\nabla(L v),\nabla v)  + v^{1-n}(L v)^2  \\
=(1-n) v^{-n}\vert\nabla v\vert^2L v+ v^{1-n} g(\nabla(\Delta_g v),\nabla v)  
+ (n-d)v^{1-n} g(\nabla g(\nabla \log|x|,\nabla v),\nabla v) + v^{1-n}(L v)^2 \,,
\end{multline}
where we also used \eqref{f-Laplacian_bis}. 
Again from \eqref{f-Laplacian_def}, we find
\begin{equation}\label{2_vera}
|x|^{d-n}\diver_g\left(\frac{1-n}{2} |x|^{n-d} v^{-n}\vert\nabla v\vert^2\nabla v\right)=\frac{1}{2} L(v^{1-n})\vert\nabla v\vert^2+ \frac{1-n}{2}v^{-n}g(\nabla v,\nabla\vert\nabla v\vert^2 )  \, ,
\end{equation}
where we used that 
\begin{equation}\label{L v 1-n}
|x|^{d-n}\diver \left(|x|^{n-d}(1-n)v^{-n}\nabla v\right)=L(v^{1-n})\, .
\end{equation}
Since
\begin{multline*}
|x|^{d-n}\diver_g\left(-\frac{1}{2}|x|^{n-d}v^{1-n}\nabla\vert \nabla v\vert^2\right)\\ =\frac{(d-n)}{2}v^{1-n}g(\nabla \log|x|, \nabla\vert\nabla v\vert^2)-\frac{1-n}{2} v^{-n}g(\nabla v, \nabla\vert\nabla v\vert^2)- \frac{1}{2}v^{1-n}\Delta_g(\vert\nabla v\vert^2) \,,
\end{multline*}
from Bochner identity we obtain
\begin{multline}\label{3_vera}
|x|^{d-n}\diver_g\left(-\frac{1}{2}|x|^{n-d}v^{1-n}\nabla\vert \nabla v\vert^2\right)=\frac{(d-n)}{2}v^{1-n}g(\nabla \log|x|, \nabla\vert\nabla v\vert^2)-\frac{1-n}{2} v^{-n}g(\nabla v, \nabla\vert\nabla v\vert^2)\\
- v^{1-n}\vert H_v\vert^2-v^{1-n} g(\nabla(\Delta_g v),\nabla v)- v^{\gamma} \ric_g(\nabla v,\nabla v)\, .
\end{multline}	

It is immediate that from \eqref{L v 1-n} we have 
\begin{equation}\label{4_vera}
|x|^{d-n}\diver_g\left(\frac{2(1-n)}{n(n-2)}|x|^{n-d} v^{-n}\nabla v  \right)=\frac{2}{n(n-2)} L(v^{1-n})
\end{equation}
and that we have  
\begin{equation}\label{5_vera}
\frac{(d-n)}{2}g(\nabla \log |x|, \nabla\vert\nabla v\vert^2)-(d-n)g(\nabla(g(\nabla \log |x|, \nabla v)),\nabla v)= H (\nabla v,\nabla v)\, ,
\end{equation}
where $H$ is given by \eqref{H_g}. Finally, from \eqref{1_vera}, \eqref{2_vera}, \eqref{3_vera}, \eqref{4_vera} and \eqref{5_vera} we obtain
\begin{multline} \label{quasi}
	|x|^{d-n}\diver_g\left(v^{1-n}L v |x|^{n-d}\nabla v + \frac{1-n}{2}v^{-n}|\nabla v|^2 |x|^{n-d} \nabla v -\dfrac{1}{2} v^{1-n}|x|^{n-d}\nabla |\nabla v|^2+ v^{-n}|x|^{n-d}\frac{2(1-n)}{n(n-2)}\nabla v\right) \\ 
	=\frac{1}{2}|\nabla v|^2 L (v^{1-n})+\frac{2}{n(n-2)}L (v^{1-n}) + (1-n)v^{-n}|\nabla v|^2 Lv + \frac{n-1}{n}v^{1-n}(Lv)^2
	-v^{1-n}\textsf{k}[v]\, .
	\end{multline}
Since from  \eqref{L v 1-n} we have
	$$
	(1-n)v^{-n}|\nabla v|^2 Lv + \frac{n-1}{n}v^{1-n}(Lv)^2=-\frac{1}{n}v Lv L(v^{1-n})\,,
	$$ 
then the differential identity \eqref{diff_ident_DEL} follows from \eqref{quasi}.
\end{proof}

From Lemma \ref{lemma_diff_ident_DEL} we immediately obtain its counterpart in an integral form.

\begin{lemma} \label{lemma_integr1}
	Let $v\in C^2(E)$. Then, for any $\varphi\in C^{\infty}_c(E)$, we have 
	\begin{multline} \label{integral_intermediate}
	\int_{E}|x|^{n-d}\Big(-\frac{v}{n}\left (Lv-\frac{n}{2}\frac{|\nabla v|_g^2}{v}-\frac{2}{n-2}\frac{1}{v}\right )L (v^{1-n}) -v^{1-n}\textsf{\emph{k}}[v]\,
	\Big) \varphi \, dx\\=-\int_{E}|x|^{n-d} v^{1-n}g \Big(L v \nabla v + \frac{1-n}{2}\frac{|\nabla v|_g^2}{v}  \nabla v -\dfrac{1}{2} \nabla |\nabla v|_g^2+ \frac{2(1-n)}{n(n-2)}\frac{\nabla v}{v},\nabla \varphi \Big )\, dx \, .
	\end{multline}
\end{lemma}

\begin{proof}
Let $\varphi$ be fixed and set $U={\rm supp} (\varphi)$. By a standard approximation argument, we can consider $v^\epsilon \in C^\infty(U)$ such that $\|v^\epsilon-v\|_{C^2(U)} \to 0$ as $\epsilon \to 0^+$. The assertion immediately follows from Lemma \ref{lemma_diff_ident_DEL} applied to $v^\epsilon$, by multiplying \eqref{diff_ident_DEL} by $\varphi$, integrating by parts and then by letting $\epsilon \to 0^+$ (we recall that $n-d\geq0$).
\end{proof}

Proposition \ref{lemma_integr3} will be obtained from Lemma \ref{lemma_integr1} by an approximation argument. Indeed, if $O \in \tilde \Omega$ we cannot set $E= \tilde \Omega$ in \eqref{integral_intermediate} and get \eqref{prop_integr}, since there are issues that must be carefully taken into account. As one can expect, by approximation and by using the boundary condition $g(\nabla w, \nu)=0$, the second line in \eqref{integral_intermediate} will give the RHS of \eqref{prop_integr}. A more subtle issue is the lack of regularity of the solution (and of the equation) at the origin due to the presence of the weight. Hence we will apply Lemma \ref{lemma_integr1} by setting $E=\tilde \Omega \setminus \overline B_\epsilon$, for $\epsilon$ small enough, and then by letting $\epsilon$ to zero. In order to do this, we need the asymptotic estimates in Proposition \ref{prop_stime_cgen} below, which will be proved in Appendix \ref{appendix}. These estimates are expressed in polar coordinates $(r,\omega)\in\mathbb{R}\times\mathbb{S}^{n-1}$ and in terms of $w'$, $w''$,  $\nabla_\omega$ and $\Delta_\omega$ which denote the derivatives of $w$ with respect to the variable $r$, the gradient and the Laplacian with respect to the variable $\omega$, respectively. 

\begin{proposition}
	\label{prop_stime_cgen}
	Let $v$ be given by \eqref{def_v} and let $\tilde \Omega$ be a bounded domain containing the origin. Let $R>0$ be such that $B_R\subset \tilde\Omega$. If $\alpha \leq \alpha_{FS}$, then we have 
	\begin{enumerate}
		\item $\int_{\mathbb{S}^{d-1}}  |v'(r,\omega)|^2 d\sigma \leq   O(1)$,
		\item $\int_{\mathbb{S}^{d-1}}  |\nabla_\omega v(r,\omega)|^2 d\sigma\leq  O(r^2)$,
		\item $\int_{\mathbb{S}^{d-1}}  |\nabla_\omega v'(r,\omega)|^2 d\sigma \leq   O(1)$,
		\item  $\int_{\mathbb{S}^{d-1}}  |\nabla_\omega v'(r,\omega)-\frac 1 r \nabla_\omega v(r,\omega)|^2 d\sigma \leq O(1)$ 
	\end{enumerate}
as $r\to 0^+$.
\end{proposition}

\noindent Now we are ready to give the proof of Proposition \ref{lemma_integr3}.

\begin{proof}[Proof of Proposition \ref{lemma_integr3}]	
Let $v$ be the solution of \eqref{pb_general}. We recall that $v$ is related to the solution $u$ of \eqref{pb_general_intro} by \eqref{w_u} and \eqref{def_v}, and then it inherits regularity properties from $u$. Standard elliptic regularity theory ensures that $u$ is smooth outside the origin, and hence the same holds for $v$. This implies that $v$ is eligible to be used in Lemma \ref{lemma_integr1} whenever the domain $E$ does not contain the origin. For this reason, in this proof we assume that $O \in \tilde \Omega$, since otherwise the proof immediately follows from Lemma \ref{lemma_integr1}.

Let $0<\varepsilon\ll1$ be fixed. We apply Lemma \ref{lemma_integr1} by setting $E=\tilde\Omega\setminus \overline B_\epsilon$ and considering  $\varphi\in C^{\infty}_c(\tilde\Omega\setminus \overline B_\epsilon)$. Hence we can write
	\begin{multline}\label{formula_palla1}
	\int_{\tilde\Omega}|x|^{n-d}\Big(-\frac{v}{n}\left (Lv-\frac{n}{2}\frac{|\nabla v|_g^2}{v}-\frac{2}{n-2}\frac{1}{v}\right )L (v^{1-n}) -v^{1-n}\textsf{k}[v]\,
	\Big) \varphi \, dx\\=-\int_{\tilde\Omega}|x|^{n-d} v^{1-n}g \Big(L v \nabla v + \frac{1-n}{2}\frac{|\nabla v|_g^2}{v}  \nabla v -\dfrac{1}{2} \nabla |\nabla v|_g^2+ \frac{2(1-n)}{n(n-2)}\frac{\nabla v}{v},\nabla \varphi \Big )\, dx \, .
	\end{multline}
	 From  \eqref{formula_palla1} and \eqref{pb_general} we have 
	\begin{multline}
	\int_{\tilde\Omega}|x|^{n-d}\Big(-\frac{1}{n}\left (v\hat f(v) -\frac{2}{n-2}\right )L (v^{1-n}) -v^{1-n}\textsf{k}[v]\,
	\Big) \varphi \, dx\\ 
	= - \int_{\tilde\Omega}|x|^{n-d} v^{1-n} \Big(\hat f(v) g(\nabla v, \nabla \varphi) + \dfrac{1}{2}\dfrac{|\nabla v|_g^2}{v}    g(\nabla v, \nabla \varphi) -\dfrac{1}{2} g(\nabla |\nabla v|_g^2,\nabla \varphi)+ \frac{2(1-n)}{n(n-2)}\frac{g(\nabla v, \nabla \varphi)}{v} \Big )\, dx \,. \label{2.32}
	\end{multline}
Now we notice that, by multiplying the equation in \eqref{pb_general} by $|x|^{n-d}v^{1-n}\hat f(v)\varphi$ and integrating by parts in ${\tilde\Omega}$, we find
	\begin{equation}\label{eq_molt}
	\begin{aligned} 
	\int_{\tilde\Omega} |x &|^{n-d}
	v^{1-n} \Big(\hat f^2(v)+\dfrac{n}{2}\dfrac{|\nabla v|_g^2}{v}\hat f(v)\Big )\varphi\,dx \\
	 & = \int_{\tilde\Omega  } |x|^{n-d}v^{1-n}L  v \hat f(v) \varphi dx
	    =-\int_{\tilde\Omega  }  |x|^{n-d}g(\nabla v ,\nabla(v^{1-n}\hat f (v)\varphi))dx \\
	& = (n-1)\int_{\tilde\Omega  }  |x|^{n-d}v^{-n}|\nabla v|_g^2\hat f (v)\varphi dx -\int_{\tilde\Omega  }  |x|^{n-d} v^{1-n} |\nabla v|_g^2\hat f'(v)\varphi dx \\ & \hspace{7cm} 
	-\int_{\tilde\Omega }  |x|^{n-d}v^{1-n}\hat f(v) g(\nabla v, \nabla \varphi)dx \,.
	\end{aligned}
	\end{equation}
	Since from \eqref{pb_general} we have
\begin{equation*}
\begin{split}
L (v^{1-n}) & =  - (n-1) v^{-n}Lv + n(n-1) v^{-n-1} |\nabla v|_g^2 \\
& = - (n-1) v^{-n}\hat f(v) + \frac{n(n-1)}{2} v^{-n-1} |\nabla v|_g^2  \,,
\end{split}
\end{equation*}
from \eqref{2.32} we have
\begin{equation*}
	\begin{aligned} 
 - \int_{\tilde\Omega} & |x|^{n-d}  v^{1-n} \Big(\hat f(v)  g(\nabla v, \nabla \varphi) + \dfrac{1}{2}\dfrac{|\nabla v|_g^2}{v}    g(\nabla v, \nabla \varphi) -\dfrac{1}{2} g(\nabla |\nabla v|_g^2,\nabla \varphi)+ \frac{2(1-n)}{n(n-2)}\frac{g(\nabla v, \nabla \varphi)}{v} \Big )\, dx= \\
	&\int_{\tilde\Omega}|x|^{n-d}\Big(\frac{n-1}{n}v^{-1-n}\left (v\hat f(v) -\frac{2}{n-2}\right )\Big(v\hat f(v)-\frac n 2 |\nabla v|^2_g \Big) -v^{1-n}\textsf{k}[v]\,
	\Big) \varphi \, dx=\\
	&\int_{\tilde\Omega}|x|^{n-d}\Big(\frac{n-1}{n}v^{-1-n}\left (v^2\hat f^2(v) -\Big(\frac{2}{n-2}+\frac n 2 |\nabla v|^2_g\Big)v \hat f(v)+\frac{n}{n-2}|\nabla v|^2_g\right )-v^{1-n}\textsf{k}[v]\,
	\Big) \varphi \, dx \,,
\end{aligned}
	\end{equation*}
and by using \eqref{eq_molt} we obtain
\begin{equation*}
	\begin{aligned} 
 - \int_{\tilde\Omega} & |x|^{n-d} v^{1-n} \Big(\hat f(v) g(\nabla v, \nabla \varphi) + \dfrac{1}{2}\dfrac{|\nabla v|_g^2}{v}    g(\nabla v, \nabla \varphi) -\dfrac{1}{2} g(\nabla |\nabla v|_g^2,\nabla \varphi)+ \frac{2(1-n)}{n(n-2)}\frac{g(\nabla v, \nabla \varphi)}{v} \Big )\, dx= \\
	&\int_{\tilde\Omega }|x|^{n-d}\Big(\frac{n-1}{n}\left ( -\Big(\frac{2}{n-2}+ n  |\nabla v|^2_g\Big)v^{-n} \hat f(v)+\frac{n}{n-2}v^{-1-n}|\nabla v|^2_g\right )-v^{1-n}\textsf{k}[v]\,
	\Big) \varphi \, dx \\
	&+ \frac{(n-1)^2}{n}\int_{\tilde\Omega }  |x|^{n-d}v^{-n}|\nabla v|_g^2\hat f (v) \varphi dx -\frac{n-1}{n}\int_{\tilde\Omega }  |x|^{n-d} v^{1-n}|\nabla v|_g^2\hat f'(v) \varphi dx \\
	&-\frac{n-1}{n} \int_{\tilde\Omega}  |x|^{n-d}v^{1-n}\hat f(v) g(\nabla v, \nabla \varphi)dx,
	\end{aligned}
	\end{equation*}
	that is 
	\begin{equation*}
	\begin{aligned}
	\frac{n-1}{n}\int_{\tilde\Omega} & |x|^{n-d} v^{1-n}|\nabla v|_g^2\hat f'(v) \varphi dx 
	\\& =  \int_{\tilde\Omega}|x|^{n-d} v^{1-n} \Big( \dfrac{1}{2}\dfrac{|\nabla v|_g^2}{v}    g(\nabla v, \nabla \varphi) -\dfrac{1}{2} g(\nabla |\nabla v|_g^2, \nabla \varphi)+ \frac{2(1-n)}{n(n-2)}\frac{g(\nabla v,  \nabla \varphi)}{v} \Big )\, dx \, \\
	&+ \int_{\tilde\Omega }|x|^{n-d}\Big(\frac{n-1}{n}\left ( -\Big(\frac{2}{n-2}+   |\nabla v|^2_g\Big)v^{-n} \hat f(v)+\frac{n}{n-2}v^{-1-n}|\nabla v|^2_g\right )-v^{1-n}\textsf{k}[v]\,
	\Big)\varphi \, dx \\
	&+\frac{1}{n}\int_{\tilde\Omega}  |x|^{n-d}v^{1-n}\hat f(v) g(\nabla v,  \nabla \varphi)dx. 
	\end{aligned}
	\end{equation*}
	Therefore, since from \eqref{Phi_def} and \eqref{hat_f_v} we have 
	$$
	\hat f'(v) = - \frac{2}{n-2} \frac{\Phi(v^{-\frac{n-2}{2}})}{v^2} - \frac{\Phi'(v^{-\frac{n-2}{2}})}{v^{1+\frac{n}{2}}},
	$$
	we obtain 
	\begin{equation}\label{2.50}
	\begin{aligned}
	-\frac{n-1}{n}\int_{\tilde\Omega} & |x|^{n-d} v^{-\frac32 n}|\nabla v|_g^2 \Phi'\left(v^{-\frac{n-2}{2}}\right)  \varphi  dx \\
	 & = \int_{\tilde\Omega}|x|^{n-d} v^{1-n} \Big( \dfrac{1}{2}\dfrac{|\nabla v|_g^2}{v}    g(\nabla v, \nabla \varphi) -\dfrac{1}{2} g(\nabla |\nabla v|_g^2,\nabla  \varphi)+ \frac{2(1-n)}{n(n-2)}\frac{g(\nabla v, \nabla  \varphi)}{v} \Big )\, dx \,\\
	&+\int_{\tilde\Omega}|x|^{n-d}\Big(\frac{n-1}{n}\left ( -\frac{4}{(n-2)^2}v^{-1-n}  \Phi\left(v^{-\frac{n-2}{2}}\right)+\frac{n}{n-2}v^{-1-n}|\nabla v|^2_g\right )-v^{1-n}\textsf{k}[v]\,
	\Big) \varphi \, dx \\
	& +\frac{2}{n(n-2)} \int_{\tilde\Omega}  |x|^{n-d}v^{-n} \Phi\left(v^{-\frac{n-2}{2}}\right) g(\nabla v, \nabla \varphi)dx.
	\end{aligned}
	\end{equation}
Now we notice that by multiplying the equation in \eqref{pb_general} by $|x|^{n-d}v^{-n}\varphi$ and integrating by parts, we obtain
\begin{multline}	\label{eq_int2}
\int_{\tilde\Omega } |x|^{n-d}
v^{-n}\Big(\frac{2}{n-2} \frac{\Phi\left(v^{-\frac{n-2}{2}}\right)}{v}+\dfrac{n}{2}\dfrac{|\nabla v|_g^2}{v}\Big ) \varphi \,dx  =-\int_{\tilde\Omega }  |x|^{n-d}g(\nabla v ,\nabla(v^{-n}\varphi))dx \\
 =n\int_{\tilde\Omega }  |x|^{n-d}v^{-1-n}|\nabla v|_g^2\varphi dx - \int_{\tilde\Omega}  |x|^{n-d}v^{-n} g(\nabla v, \nabla \varphi)dx \,.
\end{multline}
Hence, from \eqref{2.50} and \eqref{eq_int2} we finally have 
\begin{multline} \label{daqua}
	-\frac{n-1}{n}\int_{\tilde\Omega }  |x|^{n-d} v^{-\frac32 n}|\nabla v|_g^2 \Phi'\left(v^{-\frac{n-2}{2}}\right) \varphi dx 
=\frac{1}{2}\int_{\tilde\Omega}|x|^{n-d} v^{1-n}  \dfrac{|\nabla v|_g^2}{v}    g(\nabla v, \nabla \varphi) dx \\
	-\frac{1}{2} \int_{\tilde\Omega} |x|^{n-d} v^{1-n} g(\nabla |\nabla v|_g^2,\nabla \varphi) \, dx \, 
	-\int_{\tilde\Omega}|x|^{n-d}v^{1-n}\textsf{k}[v]\,
	\varphi \, dx  \\ +\frac{2}{n(n-2)}\int_{\tilde\Omega}  |x|^{n-d}v^{-n}\Phi\left(v^{-\frac{n-2}{2}}\right) g(\nabla v, \nabla \varphi)dx 
	+\frac{2}{ n(n-2)}\int_{\tilde\Omega}  |x|^{n-d}v^{-n}g(\nabla v, \nabla \varphi) \, dx 
\end{multline}
for any $\varphi\in C^{\infty}_c(\tilde\Omega\setminus \overline B_\epsilon)$.

Let $\delta, \varepsilon>0$ be small enough and define 
$$
\tilde \Omega_\delta = \{x \in \tilde \Omega:\  \dist(x,\partial  \tilde\Omega) \geq \delta  \}\,.
$$ Up to a standard density argument, we can choose $\varphi$ in \eqref{daqua} as follows: 
\begin{equation*}
\varphi(x) = 
\begin{cases}
0 &   |x| \leq \varepsilon  \\
\frac1 \delta (|x|-\varepsilon) &  \varepsilon < |x| < \varepsilon + \delta \\
1 &   \varepsilon + \delta \leq |x| \textmd{ and } x \in \tilde \Omega_\delta \\
\frac{1}{\delta} \dist(x,\partial  \tilde\Omega) &  x \in \tilde \Omega \setminus \tilde \Omega_\delta  \\
0 & x \not \in \tilde\Omega \,.
\end{cases}
\end{equation*}
By letting $\delta \to 0^+$, from \eqref{daqua} we obtain 
\begin{multline} \label{bordo}
	-\frac{n-1}{n}\int_{\tilde\Omega \setminus B_\varepsilon}  |x|^{n-d} v^{-\frac32 n}|\nabla v|_g^2 \Phi'\left(v^{-\frac{n-2}{2}}\right)  dx + \int_{\tilde\Omega \setminus B_\varepsilon}|x|^{n-d}v^{1-n}\textsf{k}[v] \, dx\\
=  - \frac{1}{2}\int_{\partial (\tilde\Omega \setminus B_\varepsilon)}|x|^{n-d} v^{-n}  |\nabla v|_g^2    g(\nabla v, \nu_g) d\sigma  
	+ \frac{1}{2} \int_{\partial (\tilde\Omega \setminus B_\varepsilon)} |x|^{n-d} v^{1-n} g(\nabla |\nabla v|_g^2,\nu_g) \, d\sigma  \, 
  \\ - \frac{2}{n(n-2)}\int_{\partial (\tilde\Omega \setminus B_\varepsilon)}  |x|^{n-d}v^{-n}\Phi\left(v^{-\frac{n-2}{2}}\right) g(\nabla v, \nu_g)d\sigma 
	- \frac{2}{n(n-2)}\int_{\partial (\tilde\Omega \setminus B_\varepsilon)}  |x|^{n-d}v^{-n}g(\nabla v, \nu_g) \, d\sigma  \,,
\end{multline}
where $\nu_g$ is the outward normal vector to $\tilde \Omega \setminus \overline B_\varepsilon$. 

In order to deal with the term $g(\nabla |\nabla v|_g^2,\nu_g)$, we consider $\{e_1,\ldots, e_{d-1}, e_d\}$ to be a local orthonormal frame such that $\{e_1,\ldots, e_{d-1}\}$ are tangent to $\partial(\tilde\Omega \setminus B_\varepsilon)$ at $q \in \partial (\tilde\Omega \setminus B_\varepsilon)$ and $e_d$ is the outward normal vector. For two vectors $X,Y$ tangent to $\partial(\tilde\Omega \setminus B_\varepsilon)$, we define the second fundamental form of $\partial(\tilde\Omega \setminus B_\varepsilon)$ by $\mathrm{II}(X;Y)=g(D_Xe_d,Y)$, where $D_X$ is the covariant derivative of the Riemannian connection on $(R^d,g)$ and the second fundamental form is chosen in such a way that spheres have positive mean curvature. For two vectors $a,b$ in the tangent space to $(R^d,g)$ at a point $x$, the Hessian tensor of $v$ is given by $H_g v (a,b) = a(bv)-(D_ab)v$. With this notation, we have that
\begin{equation*}
\begin{split}
\frac12 g(\nabla |\nabla v|_g^2,\nu_g) & = \sum_{i=1}^{d-1} v_{id} v_i \sum_{i=1}^{d-1} [e_i(e_d v)-(D_{e_i}e_d)v] v_i = g(\nabla^T\partial_{\nu_g} v, \nabla^Tv) - \mathrm{II}(\nabla^Tv,\nabla^Tv) \,,
\end{split}
\end{equation*}
where $\nabla^Tv$ is the tangential gradient of $v$ on $\partial \tilde \Omega$; hence we have
\begin{multline} \label{bordo}
	-\frac{n-1}{n}\int_{\tilde\Omega \setminus B_\varepsilon}  |x|^{n-d} v^{-\frac32 n}|\nabla v|_g^2 \Phi'\left(v^{-\frac{n-2}{2}}\right)  dx + \int_{\tilde\Omega \setminus B_\varepsilon}|x|^{n-d}v^{1-n}\textsf{k}[v] \, dx \\
=  - \frac{1}{2}\int_{\partial (\tilde\Omega \setminus B_\varepsilon)}|x|^{n-d} v^{-n}  |\nabla v|_g^2    g(\nabla v, \nu_g) d\sigma \\  
+ \int_{\partial (\tilde\Omega \setminus B_\varepsilon)} |x|^{n-d} v^{1-n} g(\nabla^T\partial_{\nu_g} v, \nabla^Tv) \, d\sigma  \, 
	- \int_{\partial (\tilde\Omega \setminus B_\varepsilon)} |x|^{n-d} v^{1-n} \mathrm{II}(\nabla^Tv,\nabla^Tv)  \, d\sigma  \, \\
  \\ - \frac{2}{n(n-2)}\int_{\partial (\tilde\Omega \setminus B_\varepsilon)}  |x|^{n-d}v^{-n}\Phi\left(v^{-\frac{n-2}{2}}\right) g(\nabla v, \nu_g)d\sigma 
	- \frac{2}{n(n-2)}\int_{\partial (\tilde\Omega \setminus B_\varepsilon)}  |x|^{n-d}v^{-n}g(\nabla v, \nu_g) \, d\sigma  \,,
\end{multline}

We first consider the boundary term on $\partial \tilde \Omega$ appearing on the RHS of \eqref{bordo}, that is
\begin{multline} \label{J_RHS}
J_{\partial \tilde \Omega}:=  - \frac{1}{2}\int_{\partial \tilde \Omega}|x|^{n-d} v^{-n}  |\nabla v|_g^2    g(\nabla v, \nu_g) d\sigma  \\
	+ \int_{\partial \tilde\Omega } |x|^{n-d} v^{1-n} g(\nabla^T\partial_{\nu_g} v, \nabla^Tv) \, d\sigma  \, 
	- \int_{\partial \tilde\Omega } |x|^{n-d} v^{1-n} \mathrm{II}(\nabla^Tv,\nabla^Tv)  \, d\sigma  \,
  \\ - \frac{2}{n(n-2)}\int_{\partial \tilde \Omega}  |x|^{n-d}v^{-n}\Phi\left(v^{-\frac{n-2}{2}}\right) g(\nabla v, \nu_g)d\sigma 
	- \frac{2}{n(n-2)}\int_{\partial \tilde \Omega}  |x|^{n-d}v^{-n}g(\nabla v, \nu_g) \, d\sigma  \,.
\end{multline}
Since $\partial_{\nu_g} v = 0$ on $\partial \tilde \Omega$, we have
\begin{equation} \label{Jleq0}
J_{\partial \tilde \Omega} = - \int_{\partial \tilde \Omega} |x|^{n-d} v^{1-n} \mathrm{II}(\nabla^Tv,\nabla^Tv) \, d\sigma \,.
\end{equation}
	
Now we consider the boundary term on $\partial B_\varepsilon$ appearing on the RHS of \eqref{bordo}, that is
\begin{multline} \label{J_RHS_Be}
J_{\partial B_\varepsilon}:=  - \frac{1}{2}\int_{\partial B_\varepsilon}|x|^{n-d} v^{1-n}  \dfrac{|\nabla v|_g^2}{v}    g(\nabla v, \nu_g) d\sigma  \\
	+ \int_{\partial B_\varepsilon } |x|^{n-d} v^{1-n} g(\nabla^T\partial_{\nu_g} v, \nabla^Tv) \, d\sigma  \, 
	- \int_{\partial B_\varepsilon } |x|^{n-d} v^{1-n} \mathrm{II}(\nabla^Tv,\nabla^Tv)  \, d\sigma  \,
	  \\ - \frac{2}{n(n-2)}\int_{\partial B_\varepsilon}  |x|^{n-d}v^{-n}\Phi\left(v^{-\frac{n-2}{2}}\right) g(\nabla v, \nu_g)d\sigma 
	- \frac{2}{n(n-2)}\int_{\partial B_\varepsilon}  |x|^{n-d}v^{-n}g(\nabla v, \nu_g) \, d\sigma  \,,
\end{multline}
and we show that $J_{\partial B_\varepsilon}$ vanishes as $\varepsilon \to 0^+$. We notice that in this case $\nu_g$ is the inner normal to $B_\varepsilon$. In order to estimate $J_{\partial B_\varepsilon}$ it will be useful to write the integrals in polar coordinates, and denote by $v' $ the radial derivative of $v$. 

We first notice that, since $u$ is bounded, say $0<u \leq C$, then from \eqref{w_u} and \eqref{def_v} we obtain that $v\geq C^{-\frac{2}{n-2}}$ and then $v^{-n}$ and $v^{1-n}$ are bounded from above. Moreover, from \eqref{Phi_def} we also have that $\Phi(v)$ is bounded.  Hence, from \eqref{J_RHS_Be} we have that there exists $C'>0$, depending only on the dimension and the $C^0$ norm of $v$, such that
\begin{equation*}
\begin{split}
|J_{\partial B_\varepsilon}| & \leq C'\varepsilon^{n-1} \int_{\mathbb{S}^{d-1}}  \Big ( |\nabla v|_g^3 + |\nabla^T v'|_g |\nabla^Tv|_g + \frac{1}{\varepsilon} |\nabla^Tv|^2_g + |\nabla v|_g \Big ) \, d\sigma \\
& \leq C''\varepsilon^{n-1} \int_{\mathbb{S}^{d-1}}  \Big ( |\nabla v|_g^3 + |\nabla^T v'|_g^2+ \frac{1}{\varepsilon} |\nabla^Tv|^2_g +  |\nabla v|_g\Big ) \, d\sigma 
\end{split}
\end{equation*}
where we have also used Young's inequality, that $|x| = \varepsilon$ and that $|\mathrm{II}(\nabla^Tv,\nabla^Tv)| \leq \varepsilon^{-1} |\nabla^Tv|^2_g$ on $\partial B_\varepsilon$.

From Lemma \ref{lemma_grad} below, we know that $|\nabla v|_g \leq C |x|^{-1}$ as $x \to O$, and then can write
\begin{equation*}
\begin{split}
|J_{\partial B_\varepsilon}|  & \leq C'''\varepsilon^{n-1} \int_{\mathbb{S}^{d-1}}  \Big ( \frac{1}{\varepsilon }(|\nabla v|_g^2  + |\nabla^Tv|^2_g + 1)+|\nabla^T v'|^2_g   \Big ) \, d\sigma \\
& \leq C_*\left[ \varepsilon^{n-2} \int_{\mathbb{S}^{d-1}}( |\nabla v|_g^2 +1) \, d\sigma+ \varepsilon^{n-1} \int_{\mathbb{S}^{d-1}} |\nabla^T v' |^2_g  \, d\sigma \right] \,.
\end{split}
\end{equation*}
Since $\nabla^T v =\nabla_\omega v$, $|\nabla^T v|^2_g =r^{-2} |\nabla_\omega v|^2$ and 
$$
|\nabla v|_g^2 = \alpha^2 (v')^2 + \frac{|\nabla_\omega v|^2}{r^2} \,,
$$
from Proposition \ref{prop_stime_cgen} and $n \geq d \geq 3$, we obtain that $J_{\partial B_\varepsilon}\to 0$ as $\varepsilon \to 0$. From \eqref{bordo} and \eqref{Jleq0} we get
	\begin{multline*}
	\frac{n-1}{n}\int_{\tilde\Omega}  |x|^{n-d} v^{-\frac32 n}|\nabla v|_g^2 \Phi'\left(v^{-\frac{n-2}{2}}\right)  dx \\ =
	\int_{\tilde\Omega}|x|^{n-d}v^{1-n}\textsf{k}[v]\,
	\, dx + \int_{\partial \tilde\Omega}|x|^{n-d} v^{1-n} \mathrm{II}(\nabla^T v,\nabla^T v) \, dx \, ,
	\end{multline*}		
which completes the proof.	
\end{proof}

It remains to prove the following lemma which has been used in the proof of Proposition \ref{lemma_integr3}.

\begin{lemma} \label{lemma_grad}
Let $v$ be given by \eqref{def_v}, i.e. $v = w^{-\frac{2}{n-2}}$ where $w$ satisfies \eqref{pb_w}, and assume that $O \in \tilde \Omega$. Then we have 
$$
|\nabla v|_g \leq C|x|^{-1}
$$
where $C$ depends only on $\|w\|_{L^{\infty}(\tilde \Omega)}$ and $\|f\|_{C^0([0,\|w\|_\infty])}$.
\end{lemma}

\begin{proof}
Let $\rho>0$ be such that $B_\rho \subset \tilde\Omega$. For $\mu \in (0, \rho/4)$ and $x \in E:= B_4 \setminus \overline B_1$, we define $\psi_\mu(x)=w(\mu x)$. From \eqref{pb_w} we have that 
$$
L \psi_\mu(x) = - \mu^2 f(\psi_\mu(x)) \quad \text{for } x \in E \,.
$$
Since $C^{-1} \leq w \leq C$ then $\mu^2 f(\psi_\mu(x))$ is bounded uniformly with respect to $\mu$. Since $|x|>1$ then the operator $L$ is uniformly elliptic, and elliptic regularity estimates yield that $|\nabla \psi_\mu(x)| \leq K$ in $B_3 \setminus \overline B_2$, where $K$ does not depend on $\mu$ and depends only on $\|w\|_{L^{\infty}(\tilde \Omega)}$ and $\|f\|_{C^0([0,\|w\|_\infty])}$. Since the metric $g$ is zero-homogeneous, then 
$$
|\nabla w(x)| \leq \frac{K}{\mu}
$$ 
for $x \in B_3 \setminus \overline B_2$. Since $v = w^{-\frac{2}{n-2}}$ and by letting $\mu$ vary in $(0, \rho/4)$ we conclude. 
\end{proof}

\section{Proof of the main theorems} \label{section_proofs}
In this section we give the proof of Theorems \ref{teo2}, \ref{teo1} and \ref{teo_Riem}. In both cases, we start from \eqref{prop_integr_ineq} and need to prove the reverse inequality. Before proving the main theorems, we need the following lemma which compares the second fundamental form of $\partial \Omega$ in the Euclidean space $\R^d$ and the second fundamental form of $\partial \tilde \Omega$ with respect to the new metric $g$, which we obtained after the mapping $\tilde \Omega = T (\Omega)$.

\begin{lemma} \label{lemma_diff_geo}
Let $T$ and $g$ be given by \eqref{T_def} and \eqref{g_def}, respectively. Let $\Omega \subset \mathbb{R}^d$ and set $\tilde \Omega=T(\Omega)$.
Let $\mathrm{II_{\partial \Omega}}$ and $\mathrm{II}^g_{\partial \tilde \Omega}$ be the second fundamental forms of $\partial \Omega$ and $\partial \tilde \Omega$, respectively. Then, for any $x \in \partial  \Omega$, we have 
\begin{equation} \label{II_II_g}
\mathrm{II}^{g}_{\partial \tilde \Omega} =\vert x\vert^{\alpha-1}\left(\mathrm{II}_{\partial \Omega}+ (\alpha-1) \frac{x \cdot \nu}{|x|^2} g_E \right)\, ,
\end{equation}
where $\mathrm{II}^{g}_{\partial \tilde \Omega} $ and $\mathrm{II}_{\partial \Omega}$ are evaluated at $y=T(x)$ and $x$, respectively, $g_E$ denotes the Euclidean metric and $\nu$ is the Euclidean outward normal to $\partial \Omega$ at $x$.  
\end{lemma}

\begin{proof}
Let $T^{-1}$ be the inverse of $T$, i.e.
$$
T^{-1} y = |y|^{\frac{1}{\alpha} -1} y \,.
$$
We first notice that $g$ is conformal to the pull-back metric $(T^{-1})^*g_E$, where $g_E$ denotes the Euclidean metric. Indeed, by letting $y=Tx$ and 
$$
\hat g : = (T^{-1})^*g_E  \,,
$$
we have that
$$
\hat g = |y|^{2\left(\frac{1}{\alpha}-1\right)} g\,, 
$$
where we recall that 
$$
g_{ij} =  \delta_{ij} + \left(\frac{1}{\alpha^2} -1 \right) \frac{y_iy_j}{|y|^2}  \,.
$$
Indeed,  in coordinates we have 
$$
g_E=\sum_i dx_i\otimes dx_i\, , 
$$
and 
\begin{equation}\label{metric_pb}
\hat g=(T^{-1})^*g_E=\sum_i d(T^{-1}y)_i\otimes d(T^{-1}y)_i\, . 
\end{equation}
Since
\begin{equation*}
d(T^{-1}y)_i=|y|^{\frac{1}{\alpha} -1} dy_i+\left(\frac{1}{\alpha}-1\right)\sum_j|y|^{\frac{1}{\alpha} -3}y_jy_i dy_j
=\sum_j|y|^{\frac{1}{\alpha} -1} \left[\delta_{ij}+ \left(\dfrac{1}{\alpha}-1\right)\dfrac{y_iy_j}{\vert y\vert^2}\right] dy_j\, ,
\end{equation*}
from \eqref{metric_pb} we have 
\begin{align*}
\hat g &=  |y|^{2\left(\frac{1}{\alpha} -1\right)}\left[\delta_{ij}+ \left(\dfrac{1}{\alpha}-1\right)\dfrac{y_iy_j}{\vert y\vert^2}\right] \left[\delta_{ik}+ \left(\dfrac{1}{\alpha}-1\right)\dfrac{y_iy_k}{\vert y\vert^2}\right] dy_j\otimes dy_k \\
&=  |y|^{2\left(\frac{1}{\alpha} -1\right)}\left[\delta_{jk}+ \left(\frac{1}{\alpha^2}-1\right)\frac{y_j y_k}{\vert y\vert^2} \right] dy_j\otimes dy_k \\
&= |y|^{2\left(\frac{1}{\alpha} -1\right)}g_{jk} dy_j\otimes dy_k \,.
\end{align*}

Since $T^{-1}$ is an isometry between $(\R^d,\hat g)$ and $(\R^d, g_E)$, then the second fundamental forms $\mathrm{II}_{\partial \tilde \Omega }^{\hat g}$ and $\mathrm{II}_{\partial \Omega }$ are equivalent, and \eqref{II_II_g} follows by writing how the second fundamental form changes under the conformal transformation
$$
g= |y|^{-2\left(\frac{1}{\alpha}-1\right)} \hat g \,.
$$
For simplicity we write 
\begin{equation}\label{conforme}
g= \phi^k \hat g \,, \quad \text{ where $\phi(y):=\vert y\vert$ and $k=-2\left(\frac{1}{\alpha}-1\right)$}\, . 
\end{equation}
It is known that, by operating a conformal change of metric as in \eqref{conforme}, the second fundamental form becomes
\begin{equation}\label{sff}
\mathrm{II}_{\partial \tilde \Omega }^{g}=\phi^{\frac{k}{2}}\left(\mathrm{II}_{\partial \tilde \Omega }^{\hat{g}}+ \frac{k}{2}\frac{\nabla^{\hat{g}}_{\nu_{\hat{g}}}\phi}{\phi}\hat{g}\right)\, ,
\end{equation}
where $\nu_{\hat{g}}$ denotes the outward unit normal in the metric $\hat g$ at the point $y=T(x)$. Since $\hat g = (T^{-1})^*g_E$ then
$$
\nabla^{\hat{g}}_{\nu_{\hat{g}}}\phi=\hat{g}(\nabla^{\hat{g}}\phi,\nu_{\hat{g}})=(T^{-1})^*g_E(\nabla^{\hat{g}}\phi,\nu_{\hat{g}})=g_E((T^{-1})_*\nabla^{\hat{g}}\phi, (T^{-1})_*\nu_{\hat{g}})=g_E(\nabla(\phi\circ T),\nu)
$$
(see \cite[Section 5]{Lee}). From \eqref{conforme} we get 
$$
\nabla^{\hat{g}}_{\nu_{\hat{g}}}\phi (y)=g_E(\nabla(\phi(\vert x\vert^{{\alpha}-1}x)),\nu)=g_E(\nabla \vert x\vert^{\alpha},\nu)=\alpha |x|^{\alpha-1} g_E\left( \frac{x}{|x|}, \nu \right)\, .
$$
Summing up, since $|y|=|x|^\alpha$, then \eqref{sff} becomes
$$
\mathrm{II}_{\partial \tilde \Omega }^{g}=\vert x\vert^{\alpha-1}\left(\mathrm{II}_{\partial \Omega}+ (\alpha-1) \frac{1}{|x|} g_E\left( \frac{x}{|x|}, \nu \right) g_E \right)\, ,
$$
where we used the fact that $T^{-1}$ is an isometry and $\mathrm{II}_{\partial \Omega}$ and $\mathrm{II}^{\hat g}_{\partial \tilde \Omega }$ are equivalent. 
\end{proof}

Lemma \ref{lemma_diff_geo} has some relevant consequences. Indeed, if we consider a domain $\Omega$ which is convex in the Euclidean space, then the set $\tilde \Omega = T(\Omega)$ is not necessarily convex with respect to the new metric $g$, unless it satisfies the condition \eqref{condTeo1} (see also Section \ref{sect_remarks} for more details).

We first prove Theorem \ref{teo2}, where $\Omega$ is a ball centered at the origin.

\begin{proof}[Proof of Theorem \ref{teo2}]
We first notice that convexity is preserved whenever $\Omega=B_R$ is a ball of radius $R$ centered at the origin, since in this case we have
$$
\mathrm{II}_{\partial  \tilde B_R}^g = R^{\alpha-1} \left(\mathrm{II}_{\partial B_R} + (\alpha-1) \frac{x \cdot \nu}{|x|^2} g_E \right)= \alpha R^{\alpha-2}  g_E  \,. 
$$
Hence, in this case, we have $\tilde \Omega=B_{\tilde R}$, with $\tilde R=R^\alpha$, and from Lemma \ref{lemma_diff_geo} we obtain that $\tilde \Omega$ is convex in the metric $g$.

Let $u$ be the solution of \eqref{pb_palla} and let $v$ be given by \eqref{def_v} which is a solution of \eqref{pb_general}.  We note that	
From \eqref{prop_integr_ineq} we have 
\begin{equation}\label{k<0}
\int_{B_{\tilde R}}|x|^{n-d}v^{1-n}\textsf{k}[v]\,
\, dx \leq 0.
\end{equation}
Hence, by \cite[Lemmas 5.1 and 5.2 and Remark 4]{DEL}, 
\begin{multline}\label{k>0}
\int_{B_{\tilde R}}|x|^{n-d}v^{1-n}\textsf{k}[v] \,dx\geq \alpha^2 \left (1-\frac 1 n \right) \int_{B_{\tilde R}}|x|^{n-d}v^{1-n} \left[ v''-\frac{v'}{r} - \frac{\Delta_\omega v}{\alpha^2 (n-1)r^2}\right]^2\, dx\\
+2\alpha^2\int_{B_{\tilde R}}\frac{1}{r^2} \left |\nabla_\omega v'-\frac{\nabla_\omega v}{r}\right |^2\, dx +(n-2)\left (\frac{d-1}{n-1}-\alpha^2 \right)\int_{B_{\tilde R}}|x|^{n-d}v^{1-n}\frac{1}{r^4}|\nabla_\omega v|^2\,dx.
\end{multline}
In particular, from \eqref{alpha_optimal} $$\int_{B_{\tilde R}}|x|^{n-d}v^{1-n}\textsf{k}[v]\, dx\geq 0.$$ Therefore,  $\textsf{k}[v]=0$ by \eqref{k<0}. From \eqref{k>0} we get 
$$
\nabla_\omega v=0 \quad \text{and} \quad v''-\frac{v'}{r}=v''-\frac{v'}{r} - \frac{\Delta_\omega v}{\alpha^2 (n-1)r^2}=0 \, . 
$$ 
These conditions imply that either 
$$v(r,\omega)=c+\frac{r^2}{\alpha^2\lambda}, \qquad 0\leq r<\tilde R\, , \omega \in \mathbb{S}^{d-1}\, ,
$$ 
for some positive constants $\lambda$ and $c$ or $v$ is constant.  
We exclude the first case thanks to the boundary condition in problem \eqref{pb_general}, by observing that 
$$
g(\nabla v,\nu_g)=\alpha\partial_r v=\frac{2}{\lambda\alpha} r \neq 0 \quad \text{ on $\partial\tilde\Omega$}\,.
$$
Hence we conclude that $v$ must be constant and the same holds for $u$.
\end{proof}

In the case of Theorems \ref{teo1} and \ref{teo_Riem} we cannot make use of \cite[Corollary 5.4]{DEL} since the domain is not a ball and the integral inequalities on $\partial B_r$ used in \cite[Corollary 5.4]{DEL} cannot be considered up to the boundary of $\tilde \Omega$. In this case, we exploit a pointwise estimate and prove that $\textsf{k}[v] \geq 0$, which holds under the stronger assumption on $\alpha$ given by \eqref{hp_alfa}.

\begin{proof}[Proof of Theorem \ref{teo_Riem}]
	
Let $w$ be the solution of \eqref{pb_Riem} and $v$ be given by \eqref{def_v}, which is a solution of \eqref{pb_general}. By contradiction, let us assume that $w$ (and hence $v$) is not constant. Since $v$ is not constant, from Proposition \ref{lemma_integr3}, \eqref{Phi_def} and by using the convexity of $\tilde \Omega$, we have that ${k}[v] = 0$ and $\Phi' (v^{-\frac{n-2}{2}}) = 0$, which imply that the Hessian of $v$ is a multiple of the metric $g$ and $f(v)=c v^{\frac{n+2}{n-2}}$ for some constant $c$, respectively.

	By Cauchy-Schwarz inequality,  
	\begin{equation}\label{CS}
	|H_v|^2\geq \frac{1}{d}(\Delta_g v)^2 \,,
	\end{equation}
 \eqref{ric_g}, \eqref{H_g} and the fact that $L v=\Delta_g v +(n-d)g(\nabla \log \vert x\vert,\nabla v)$,  we have 
\begin{equation*}
\begin{split}
\textsf{k}[v]& =|H_v|^2
	+\ric_g(\nabla v,\nabla v)+ H(\nabla v,\nabla v)-\frac{1}{n}(Lv)^2\\ 
	& \geq \frac{d-2-\alpha^2(n-2)}{|x|^2}\left( |\nabla v|^2-\frac{(\nabla v \cdot x)^2 }{|x|^2} \right )\geq 0
\end{split}
\end{equation*}
by \eqref{hp_alfa}. Hence, from \eqref{hp_alfa} we obtain that  $v$ is radial, which implies that $v$ is quadratic
$$
v(r,\omega)=c+\frac{r^2}{\alpha^2\lambda} \,.
$$ 
From the boundary condition $g(\nabla w, \nu_g)=0$ on $\partial \tilde\Omega$ we get a contradiction.
\end{proof}

\medskip 

As we already mentioned in the introduction, Theorem \ref{teo1} is a straightforward consequence of Theorem \ref{teo_Riem}.

\begin{proof}[Proof of Theorem \ref{teo1}]
Since \eqref{condTeo1} is in force, then $\tilde \Omega$ is convex with respect to the metric $g$ and Theorem \ref{teo1} immediately follows from Theorem \ref{teo_Riem}.
\end{proof}

%
%
%
%

\section{Further remarks} \label{sect_remarks}
In this section we exploit the geometric condition \eqref{condTeo1} given in Theorem \ref{teo1}, which we recall it is given by  
\begin{equation*} 
\mathrm{II}_{\partial \Omega}\geq (1-\alpha) \frac{x \cdot \nu}{|x|^2}  \,,
\end{equation*}
and provide some example. 

\begin{lemma}  \label{lemma_exB}
Let $\Omega=B_R(x_0)$ and $0<\alpha \leq 1$. Then $\Omega$ satisfies \eqref{condTeo1} if and only if $ |x_0| \leq \alpha R$ or $|x_0|>R$.
\end{lemma}
\begin{proof}
The proof of this lemma follows by some simple computations. If $x_0=O$ then it is clear that \eqref{condTeo1} is equivalent to 
$$
\frac{1}{R} \geq \frac{1-\alpha}{R} 
$$ 
at any point on $\partial B_R$ and then \eqref{condTeo1} is satisfied since $0<\alpha \leq 1$.

If $x_0 \neq O$, then we write $x \in \partial B_R(x_0)$ by letting $x=x_0+R \omega$ with $\omega \in \mathbb{S}^{d-1}$, we have that
$$
(1-\alpha) \frac{x \cdot \nu}{|x|^2} = \frac{1-\alpha}{|x_0|} \frac{c+t}{1+c^2+2ct} \,,
$$
where we set $c=\frac{R}{|x_0|}$ and $t=\omega \cdot \frac{x_0}{|x_0|}$. Since $|t| \leq 1$, we find that
\begin{equation*}
 \frac{x \cdot \nu}{|x|^2}  \leq 
\begin{cases} 
\frac{1}{R+|x_0|}  & \text{ if } |x_0|>R \,, \\
\frac{1}{R-|x_0|}  & \text{ if } |x_0|<R \,.
\end{cases}
\end{equation*}
This implies that \eqref{condTeo1} is fulfilled if and only if $ |x_0| \leq \alpha R$ or $|x_0|>R$.
\end{proof}

We emphasize that Lemma \ref{lemma_exB} implies that a ball containing the origin and center $x_0$ does not satisfies \eqref{condTeo1} if $|x_0| > \alpha R$ and, in this case, Theorem \ref{teo1} does not apply. This example suggests that \eqref{condTeo1} is stronger than convexity, at least when $O \in \Omega$.

\begin{lemma} \label{lemmatto}
Let $0<\alpha \leq 1$ and $\Omega \subset \mathbb R^d$ be a bounded domain, with $d \geq 2$. If $\Omega$ satisfies \eqref{condTeo1} then $\Omega$ is convex. 
\end{lemma} 

\begin{proof}
We first show that \eqref{condTeo1} implies that $\Omega$ is starshaped with respect to $O$, which then immediately implies that $\Omega$ is convex. 

By contradiction, assume that $\Omega$ contains the origin but it is not starshaped. Then there exists a point $x_0 \in \partial \Omega$ such that $x_o \cdot \nu(x_0) = 0$ and a curve $\gamma(s)$ defined in a neighborhood of $s=0$ and parametrized by arc-length such that $\gamma(0)=x_0$, 
$$
\gamma'(0)= -\frac{x_0}{|x_0|}
$$
and $\gamma(s) \cdot \nu(\gamma(s)) > 0$ for $s>0$. Moreover, according to Frenet formulas, we can choose $\gamma$ such that $\gamma''(s) = k(s) \nu(s) $ and $\nu'(s)=- k(s) \gamma'(s)$. Since $\gamma(0) \cdot \gamma'(0)= |\gamma (0)| > 0$ then, by continuity and for $s$ close to zero, we have that $\gamma(s) \cdot \nu(\gamma(s))>0$. By choosing $s>0$, the mean value theorem yields
$$
0 < \gamma(s) \cdot \nu(\gamma(s)) = \gamma(s) \cdot \nu(\gamma(s)) - \gamma(0) \cdot \nu(\gamma(0)) = - k(\bar s) \gamma(\bar s ) \cdot \gamma'(\bar s) s \,,
$$
which implies that there exists $\bar s>0$ such that $k(\bar s)<0$. Since $\gamma(\bar s) \cdot \nu(\gamma(\bar s)) > 0$ then \eqref{condTeo1} leads to a contradiction.

Hence $x \cdot \nu(x) >0$ and \eqref{condTeo1} implies that $\Omega$ is convex.
\end{proof}

We mention that, by considering a small perturbation of a ball, from Lemma \ref{lemma_exB} we easily infer that there exist domains satisfying \eqref{condTeo1} which differ from a ball. As stated in Proposition \ref{propC}, if $O \not \in \overline \Omega$ then we can prove that the domain does not need to be convex. We give an example in the proof of Proposition \ref{propC}.

\begin{proof}[Proof of Proposition \ref{propC}]
The first part of the assertion has been proved in Lemma \ref{lemmatto}. Regarding the second part of the assertion, we provide two examples, which are represented in Fig. \ref{Fig1}.

We consider a ball $B$ which does not contain the origin. We consider a domain $\Omega$ such that $\partial \Omega$ coincides with $\partial B$ at any point $x \in \partial B$ where $x \cdot \nu_{\partial B} (x)  \geq 0$. Hence Lemma \ref{lemma_exB} implies that $\partial \Omega$ satisfies \eqref{condTeo1} at those points. 

Then, we modify the lower spherical cap of $\partial B$, where $x \cdot \nu_{\partial B} (x) < 0$, in such a way that $\Omega$ remains convex and $\partial \Omega$ has a flat portion. It is clear that $\Omega$ satisfies \eqref{condTeo1} at these points since the second fundamental form is nonnegative definite and $x \cdot \nu_{\partial B} (x) < 0$. Thus $\Omega$ (as well as the ball $B$) provides an example of domain satisfying \eqref{condTeo1} which is convex and not containing the origin.

Starting from $\Omega$ we can easily build a non-convex domain by doing a small perturbation of the flat portion $\partial \Omega$, since $x \cdot \nu_{\partial \Omega} \leq c_0 < 0$ for some $c_0>0$ in that region, as the one in blue in Figure \ref{Fig1}.
\end{proof}

\begin{figure}
\includegraphics[width=0.35 \textwidth]{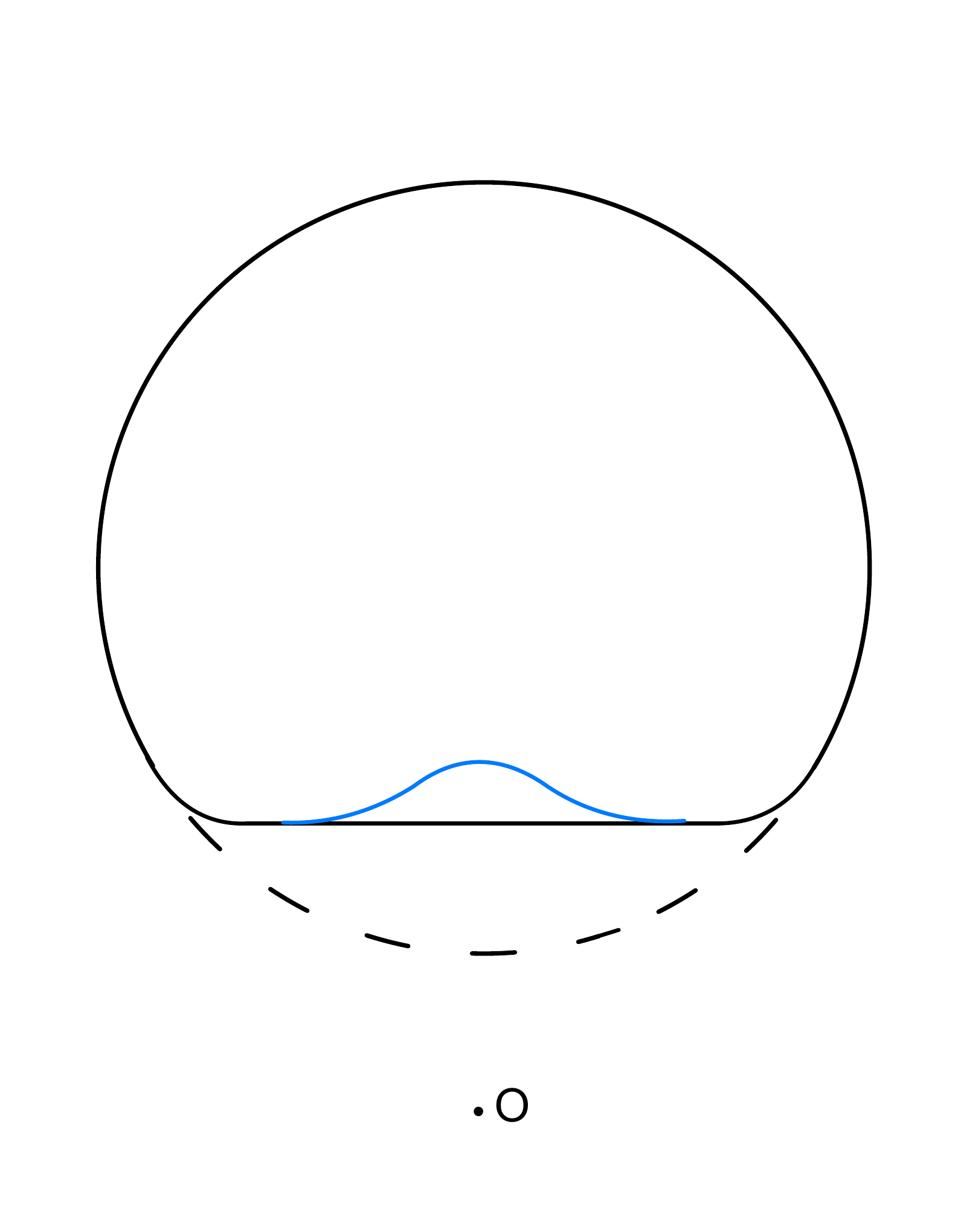}
\caption{Starting from a ball which does not contain the origin, we provide two nontrivial examples of domains satisfying \eqref{condTeo1}. The domain bounded by the black line is convex. The domain obtained with the small blue perturbation is not convex.}  \label{Fig1}
\end{figure}

\appendix

\section{Asymptotic estimates} \label{appendix}
In this appendix we prove Proposition \ref{prop_stime_cgen} which contains the regularity estimates at the origin needed to prove Proposition \ref{lemma_integr3}. The main ideas are taken from \cite{DEL}, but we prefer to give a proof of the results since some argument can be simplified in the setting that we are considering.

We recall that $u$ is a positive bounded weak solution to \eqref{eq_u_sec2}, and hence it satisfies 
\begin{equation} \label{u_weak}
\int_{ \Omega} |x|^{-bq}u^2 dx + \int_{ \Omega} |x|^{-2a}|Du|^2 dx < +\infty\,.
\end{equation}
By an extension argument and thanks to CKN inequalities, we have that $u$ is also such that
\begin{equation} \label{u_energy}
\int_{\R^d} |x|^{-bq}u^q dx < +\infty \,,
\end{equation}
with $q$ given by \eqref{parameters}. Thanks to the mapping $T:\mathbb{R}^d \to \mathbb{R}^d$, with $T(x)= |x|^{\alpha-1} x $, and by setting  
\begin{equation*} 
w(T(x)) = u(x)\,,
\end{equation*}
we have that $w$ satisfies \eqref{pb_w}, with $g\, \ L$ and $\tilde \Omega$ given by \eqref{g_inverse}, \eqref{f-Laplacian_def} and \eqref{tildeOmega_def}, respectively. From \eqref{u_weak} and \eqref{u_energy}, we also have that $w$ is such that
\begin{equation} \label{w_energy}
\int_{\tilde \Omega} w^q|x|^{n-d} dx + \int_{\tilde \Omega} |\nabla w|_g^2\, |x|^{n-d}dx < + \infty \,,
\end{equation}
As done in \eqref{def_v}, we set 
$$
v = w^{-\frac{2}{n-2}} 
$$
and we have that $v$ satisfies \eqref{pb_general}. Moreover, from \eqref{w_energy}, we also have 
\begin{equation} \label{v_energy}
\int_{\tilde \Omega} v^{-q \frac{n-2}{2} } |x|^{n-d} dx + \int_{\tilde \Omega} v^{-\frac{n}{2}}|\nabla v|_g^2\, |x|^{n-d}dx < + \infty \,.
\end{equation}

In order to prove Proposition \ref{prop_stime_cgen} we need a preliminary result. It will be convenient to use polar coordinates $(r,\omega)$, with $r=|x|$ and $\omega = x/|x|$. We let $R>0$ be such that $B_R\subset \tilde\Omega$, and  
we consider the Emden-Fowler transformation
\begin{equation} \label{uphi}
u(r,\omega)=r^{a-a_c}\varphi(s,\omega),  
\end{equation}
with
$$
0< r<R,\,\, s=-\log r,\,\, \omega\in \mathbb{S}^{d-1} \,.
$$ 
The function $\varphi$ satisfies
\begin{equation}\label{eq_phi}
-\partial_s^2 \varphi -\Delta_\omega \varphi +\Lambda \varphi = e^{-s(a_c+a-bq+2)}  f(e^{(a_c-a) s}\varphi)
\end{equation}
in 
\begin{equation*}
  \mathcal{C}:=[-\log R,+\infty ) \times \mathbb{S}^{d-1} \,,
\end{equation*}
with $\Lambda=(a-a_c)^2$. We notice that, after some simple calculations, \eqref{eq_phi} can be written as
\begin{equation}\label{eq_phi2}
-\partial_s^2 \varphi -\Delta_\omega \varphi +\Lambda \varphi = e^{-s(a_c-a) \frac{n+2}{n-2}}  f(e^{(a_c-a) s}\varphi)
\end{equation}
and by using the definition of $\Phi$ \eqref{Phi_def} we also have 
\begin{equation}\label{eq_phi3}
-\partial_s^2 \varphi -\Delta_\omega \varphi +\Lambda \varphi =  \Phi(e^{(a_c-a) s}\varphi)\varphi^\frac{n+2}{n-2} \,.
\end{equation}
We notice that, since 
\begin{equation} \label{ccc}
c^{-1} \leq u \leq c
\end{equation} 
in $\overline \Omega$ for some $c >0$, then 
\begin{equation} \label{Phi_bounded}
| \Phi(e^{(a_c-a) s}\varphi(s,\omega)) | \leq C
\end{equation}
for some positive constant $C$ and for any $(s,\omega) \in \mathcal C$. From \eqref{u_weak} and \eqref{u_energy} we also have that 
$$
\varphi \in H^1( \mathcal{C}) \,.
$$
We will denote by $\varphi'(s,\omega)$ and $\nabla_\omega \varphi (s,\omega)$ the derivative of $\varphi$ in the variable $s$ and the angular gradient of $\varphi$, respectively.

\begin{proposition} \label{stime_phi}
Let $\varphi$ be given by \eqref{uphi}. There exist two positive constants $C_1$ and $C_2$ such that	
\begin{equation} \label{bounds1}
C_1 e^{-\sqrt{\Lambda} s}\leq\varphi(s,\omega) \leq C_2 e^{-\sqrt{\Lambda} s}
\end{equation}
and
\begin{equation} \label{bounds2}
\vert\varphi'(s,\omega)\vert \, , \vert\varphi''(s,\omega)\vert\, , \vert\nabla_\omega \varphi(s,\omega)\vert\, \vert\nabla_\omega\varphi'(s,\omega)\vert\, , \vert\Delta_\omega\varphi(s,\omega)\vert\leq C_2 e^{-\sqrt{\Lambda} s}
\end{equation}
for any $(s,\omega)\in\mathcal{C}$.	
\end{proposition}

\begin{proof}
We first notice that $u > 0 $ in $ \overline \Omega$, since $u$ cannot vanish at the boundary by Hopf's lemma. Hence, there exists a positive constant $c_*$ such that $0<c_*^{-1} \leq u (x) \leq c_*$ for any $x \in \overline \Omega$. Thus \eqref{bounds1} immediately follows from \eqref{uphi}.

By a localized boot-strap argument (see e.g. \cite[Corollary 7.11, Theorem 8.10, and Corollary 8.11]{GT}) we obtain the $C^{\infty}$ regularity. From local $C^{1,\alpha}$ estimates (see e.g. \cite[Theorem 8.32, p. 210]{GT}) we get that all the first derivatives of $\varphi$ converge to $0$ with rate $e^{-\sqrt{\Lambda}s}$ as $s\rightarrow+\infty $. Moreover, from local $W^{k+2,2}$ estimates (see e.g.  \cite[Theorem 8.10, p. 186]{GT}) we get estimates of order $e^{-\sqrt{\Lambda}s}$, for $s$ large enough.  Finally, the assertion follows from \cite[Corollary 7.11, Theorem 8.10, and Corollary 8.11]{GT} and by taking $k$ large enough.
\end{proof}

We are ready to prove Proposition \ref{prop_stime_cgen}.

\begin{proposition*}	
Let $v$ be given by \eqref{def_v} and let $\tilde \Omega$ be a bounded domain containing the origin. Let $R>0$ be such that $B_R\subset \tilde\Omega$. If $\alpha \leq \alpha_{FS}$, then we have 
	\begin{enumerate}
		\item $\int_{\mathbb{S}^{d-1}}  |v'(r,\omega)|^2 d\sigma \leq   O(1)$,
		\item $\int_{\mathbb{S}^{d-1}}  |\nabla_\omega v(r,\omega)|^2 d\sigma\leq  O(r^2)$,
		\item $\int_{\mathbb{S}^{d-1}}  |\nabla_\omega v'(r,\omega)|^2 d\sigma \leq   O(1)$,
		\item  $\int_{\mathbb{S}^{d-1}}  |\nabla_\omega v'(r,\omega)-\frac 1 r \nabla_\omega v(r,\omega)|^2 d\sigma \leq O(1)$ 
	\end{enumerate}
as $r\to 0^+$.
\end{proposition*}

\begin{proof}
	The proof of this proposition is essentially the same of \cite[Proposition 8.2]{DEL}, even if in our setting we can simplify some argument  (due to the fact that for us $u$ bounded and strictly positive in $\overline \Omega$).
	
By an abuse of (evident) notations, when passing to polar coordinates we will write $v(x)=v(r,\omega)$. We recall that 
$$
v(r,\omega) = w(r,\omega)^{-\frac{2}{n-2}}\, \quad \text{and} \quad  w(r,\omega) = u(r^{\frac{1}{\alpha}}, \omega) = r^{\frac{a-a_c}{\alpha}} \varphi\left(\frac{s}{\alpha}, \omega \right) \,,
$$	
where $u$ is the solution to \eqref{eq_u_sec2} and $w,v$ and $\varphi$ are given by \eqref{w_u}, \eqref{def_v} and \eqref{uphi}, respectively. Straightforward calculations give
$$
v'(r,\omega)= - \frac{2}{\alpha(n-2)} w(r,\omega)^{-\frac{2}{n-2}}  \left( a-a_c - \frac{\varphi'\left(\frac{s}{\alpha}, \omega \right)}{\varphi\left(\frac{s}{\alpha}, \omega \right)} \right) \frac{1}{r} \,.
$$
From \eqref{ccc} we obtain that $w$ is bounded and Proposition \ref{stime_phi} yields
\begin{equation} \label{st1}
|v'(r,\omega)| \leq  e^{s} O \left( \bigg{|} a-a_c - \frac{\varphi'\left(\frac{s}{\alpha}, \omega \right)}{\varphi\left(\frac{s}{\alpha}, \omega \right)} \bigg{|} \right) 
\end{equation}
and analogously
\begin{equation} \label{st2}
|\frac{1}{r} \nabla_\omega v(r,\omega)| \leq  e^{s} O \left( \bigg{|}  \frac{\nabla_\omega \varphi\left(\frac{s}{\alpha}, \omega \right)}{\varphi\left(\frac{s}{\alpha}, \omega \right)} \bigg{|} \right) 
\end{equation}
as $r \to  0$. 
%
Since		
\begin{equation*}
	 \nabla_\omega v'(r,\omega)-\frac{1}{r} \nabla_\omega v(r,\omega) =  \frac{2}{\alpha(n-2)r} w(r,\omega)^{-\frac{2}{n-2}}   \left (  \frac{\nabla_\omega \varphi'(\frac s \alpha,\omega)}{\varphi(\frac s \alpha,\omega)} - \frac{n}{n-2} \frac{\varphi'(\frac s \alpha,\omega) \nabla_\omega \varphi(\frac s \alpha,\omega)}{\varphi(\frac s \alpha,\omega)^2} \right ),
	\end{equation*}
and since $w$ is bounded, then Proposition \ref{stime_phi} implies
\begin{equation}\label{st4}
	\left | \nabla_\omega v'(r,\omega)-\frac{1}{r} \nabla_\omega v(r,\omega)\right |\leq  e^s  O\left (\left |  \frac{\nabla_\omega \varphi'(\frac s \alpha,\omega)}{\varphi(\frac s \alpha,\omega)} - \frac{n}{n-2} \frac{\varphi'(\frac s \alpha,\omega) \nabla_\omega \varphi(\frac s \alpha,\omega)}{\varphi(\frac s \alpha,\omega)^2} \right | \right ),
	\end{equation}
which holds uniformly with respect to $\omega$. 

Now, our goal is to prove the following asymptotic expansions, which immediately imply the assertion of the proposition.
	\begin{enumerate}
		\item[(i)] $\int_{\mathbb{S}^{d-1}}  \left |\sqrt{\Lambda}-\frac{ \varphi '(\frac s \alpha,\omega)}{ \varphi(\frac s \alpha,\omega)}\right |^2 dx \leq O(e^{-2 s})$;
		\item[(ii)] $\int_{\mathbb{S}^{d-1}}  \left | \frac{\nabla_\omega  \varphi(\frac s \alpha,\omega)}{ \varphi(\frac s \alpha,\omega) }\right |^2 dx \leq O(e^{-2 s})$;
		\item[(iii)] $\int_{\mathbb{S}^{d-1}}  \left |\frac{n}{n-2} \frac{\varphi'(\frac s \alpha,\omega) \nabla_\omega \varphi(\frac s \alpha,\omega)}{\varphi(\frac s \alpha,\omega)^2}-\frac{\nabla_\omega \varphi'(\frac s \alpha,\omega)}{\varphi(\frac s \alpha,\omega)}\right |^2 dx \leq O(e^{-2 s})$;
	\end{enumerate}
	as $s\to +\infty$. The proof of (i)-(iii) strictly follows the proof of \cite[Proposition 8.2]{DEL}, with only minor changes. For this reason, we give a more detailed sketch of the proof for (i), and we omit the proofs of (ii) and (iii) which can be obtained in a similar manner. 
		
	{\it Proof of} (i). Let us consider a positive solution $\varphi$ to \eqref{eq_phi} and define on $(\log R, +\infty)$ the
	function
	$$
	\varphi_0(s)=\int_{\mathbb{S}^{d-1}}  \varphi(s,\omega)d\sigma,
	$$
	which is a solution of 
	$$
	-\varphi_0''+\Lambda \varphi_0 = \int_{\mathbb{S}^{d-1}}  \Phi(e^{(a_c-a) s}\varphi)\varphi^\frac{n+2}{n-2} d\sigma, \quad \text{in } (\log R, +\infty);
	$$
from \eqref{Phi_bounded}	and \eqref{bounds1} we obtain that $\varphi_0(s) \sim e^{-\sqrt{\Lambda }s}$ as $s\to +\infty$. We define 
$$
\Psi(s, \omega)=e^{\sqrt{\Lambda}s}(\varphi(s,\omega)-\varphi_0(s))
$$ 
and notice that	$\Psi$ solves
\begin{equation} \label{eq_Psi}
-\partial_s^2 \Psi - \Delta_\omega \Psi - 2 \sqrt{\Lambda} \partial_s \Psi = H \,,  
\end{equation}
where $H$ satisfies 
$$
|H| \leq O(e^{-2\alpha s}) \,,
$$
as $s \to + \infty$. Now, by arguing as in \cite[p. 433]{DEL}, we obtain  
	\begin{equation}
	\label{dim_i}
	\left |\frac{\partial_s \varphi(s,\omega)}{\varphi(s,\omega)}-\sqrt{\Lambda}\right |\leq C |\partial_s \Psi(s,\omega)|+O(e^{-2\alpha s}) \quad s\to +\infty \,,
	\end{equation}
and hence (i) is proved once we have the asymptotic behaviour of 	
	\begin{equation} \label{chi1def}
	\chi_1(s):=\frac{1}{2}\int_{\mathbb{S}^{d-1}} |\partial_s \Psi|^2 d\sigma \,.
	\end{equation}
In order to do this, by closely following the proof of \cite[Proposition 8.2]{DEL}, we find that $\chi_1$ satisfies 
	\begin{equation}
	\label{eqchi1}
	-\chi_1''+\frac{|\chi_1'|^2}{2\chi_1}+2\lambda_1 \chi_1 -2 \sqrt{\Lambda} \chi_1'\leq h_1
	\end{equation}
	where 
	$$
	h_1= \int_{\mathbb{S}^{d-1}} \partial_s H \partial_s \psi d \sigma \,,
	$$
	and $\lambda_1=d-1$ is the constant in the Poincaré inequality
	$$
	\int_{\mathbb{S}^{d-1}} |\nabla_\omega(\partial_s \Psi)|^2d\sigma \geq \lambda_1 \int_{\mathbb{S}^{d-1}} |\partial_s \Psi|^2 d\sigma,
	$$
	which holds since 
	$$
	\int_{\mathbb{S}^{d-1}} \partial_s \psi d\sigma = 0
	$$ 
	from the definition of $\psi$. From Cauchy-Schwarz inequality we immediately obtain that 
	$$
	|h_1(s)|\leq C e^{2\alpha s} \sqrt{\chi_1(s)} , \quad s\to +\infty.
	$$
By setting
$$
\zeta_1 = \sqrt{\chi_1}
$$
one can find that $\zeta_1$ satisfies
$$
-\zeta_1'' + \lambda_1 \zeta_1 - 2 \sqrt{\Lambda} \zeta_1' \leq \frac{h_1}{2 \zeta_1} \leq C e^{-2\alpha s} \,.
$$
Under the condition $\alpha \leq \alpha_{FS}$, one can prove that $\zeta_1(s) \leq O(e^{-\alpha s}) $ as $s \to +\infty$ and then $\chi_1(s)\leq O(e^{-2\alpha s})$ for $s\to +\infty$, which together with \eqref{dim_i} and \eqref{chi1def}, proves (i). 

{\it Proof of} (ii) and (iii). Again, the proofs of (ii) and (iii) closely follow \cite[Proof of Proposition 8.2]{DEL}. As done for (i), the only difference is in the right hand side of the differential equations, which comes from the right hand side of \eqref{eq_phi3}. However, only the asymptotic estimates of the right hand side are used in the proof and, thanks to \eqref{Phi_bounded}, the asymptotic behaviour of the terms on the right hand side that we obtain is the same as the one in \cite{DEL}. For this reason we omit the rest of the proof.

The proof of the proposition now follows from \eqref{st1}, \eqref{st2}, \eqref{st4} and 	
\begin{equation*}
	\left |  \nabla_\omega v'(r,\omega)\right |\leq  \left | \nabla_\omega v'(r,\omega)-\frac{1}{r} \nabla_\omega v(r,\omega)\right |+\left |\frac{1}{r} \nabla_\omega v(r,\omega)\right |.
	\end{equation*}
by using the estimates (i)-(ii)-(iii).
\end{proof}

\section*{Acknowledgements}
The authors are indebted to Luigi Vezzoni for the discussions they had together. 

The authors have been partially supported by the ``Gruppo Nazionale per l'Analisi Matematica, la Probabilit\`a e le loro Applicazioni'' (GNAMPA) of the ``Istituto Nazionale di Alta Matematica'' (INdAM, Italy).  A.R. has been partially supported by the PRIN 2017 project ``Direct and inverse problems for partial differential equations: theoretical aspects and applications''.

\end{document}